\documentclass[11pt]{article}

\usepackage[utf8]{inputenc}
\usepackage{lmodern}
\usepackage[T1]{fontenc}
\usepackage{amsmath}
\usepackage{amssymb}
\usepackage{amsthm}
\usepackage{enumitem}
\usepackage[pdfstartview=Fit, linktocpage]{hyperref}
\usepackage{url}		
\usepackage{mathrsfs}								
\usepackage{bbm}

\numberwithin{equation}{section}
\usepackage[margin=3cm]{geometry}

\newcommand{\N}{\mathbb{N}}		
\newcommand{\Q}{\mathbb{Q}}		
\newcommand{\R}{\mathbb{R}}		
\newcommand{\expo}{\mathrm{e}}	

\def\D{{\mathcal D}}
\def\PP{\mathbb P}

\def\EE{\mathbb E}

\renewcommand{\k}{k}				
\newcommand{\kk}{{\underline{k}}}	
\newcommand{\kkk}{\ell}				
\newcommand{\kkkk}{{\underline{\ell}}}	
\newcommand{\kkkkk}{{\overline{k}}}
\newcommand{\kkkkkk}{{\overline{\ell}}}

\newcommand{\RCD}{\mathrm{RCD}}		
\newcommand{\CD}{\mathrm{CD}}		
\newcommand{\BE}{\mathrm{BE}}		
\newcommand{\EVI}{\mathrm{EVI}}		
\newcommand{\GE}{\mathrm{GE}}  		
\newcommand{\PCP}{\mathrm{PCP}}   	
\newcommand{\PTE}{\mathrm{TE}} 		
\newcommand{\B}{\mathsf{b}}			
\newcommand{\X}{\B^x}				
\newcommand{\mms}{X}				
\newcommand{\meas}{\mathfrak{m}}	
\newcommand{\Ell}{L}				
\newcommand{\TestF}{\mathrm{TestF}}	
\newcommand{\Lip}{\mathrm{Lip}}		
\newcommand{\bdd}{\mathrm{b}}		
\newcommand{\loc}{\mathrm{loc}}		
\newcommand{\eval}{\mathsf{e}}		
\newcommand{\push}{\sharp}			
\newcommand{\Sob}{W}				
\newcommand{\Borel}{\mathscr{B}}	
\newcommand{\W}{W}

\newcommand{\Cont}{\mathrm{C}}		
\newcommand{\DELTA}{\mathbf{\Delta}}
\newcommand{\bs}{\mathrm{bs}}		

\newcommand{\green}{\mathrm{g}}		

\newcommand{\Schr}{\mathsf{P}}		
\newcommand{\ChHeat}{\mathsf{P}}	
\newcommand{\WHeat}{\mathsf{H}}		

\newcommand{\Prob}{\mathscr{P}}		
\newcommand{\Leb}{\mathscr{L}}		
\newcommand{\Exp}{\mathbb E}		
\newcommand{\One}{\mathbbm{1}}		

\newcommand{\Diff}{\mathrm{D}}		
\newcommand{\diff}{\mathrm{d}}		
\renewcommand*\d{\mathop{}\!\mathrm{d}}	

\newcommand{\PS}[1]{{\Pi_{#1}}}		

\DeclareMathOperator{\met}{\mathsf{d}}	
\DeclareMathOperator{\Ent}{Ent}			
\DeclareMathOperator{\supp}{supp}		
\DeclareMathOperator{\Hess}{Hess}		
\theoremstyle{definition}
\newtheorem{bump}{Bump}[section]
\newtheorem{remarkx}[bump]{Remark}

\newenvironment{remark}{\pushQED{\qed}\remarkx}{\popQED\endremarkx}

\theoremstyle{plain}
\newtheorem{thm}[bump]{Theorem}
\newtheorem{Prop}[bump]{Proposition}
\newtheorem{Def}[bump]{Definition}
\newtheorem{lma}[bump]{Lemma}
\newtheorem{cor}[bump]{Corollary}

\newtheoremstyle{cited}
	{\topsep}		
	{\topsep}		
	{\itshape}		
  	{}				
	{\bfseries}		
	{\textbf{.}}	
	{.5em}			
	{\thmname{#1} \thmnumber{#2} \thmnote{\normalfont#3}}		

\theoremstyle{cited}	

\newtheorem{citedefinition}[bump]{Definition}

\newcommand{\Geo}{\mathrm{Geo}}
\newcommand{\G}{\mathrm{G}}
\newcommand{\Ric}{\mathrm{Ric}}

\newcommand{\Dom}{\mathrm{Dom}}

\newcommand{\E}{\mathscr{E}}
\newcommand{\lip}{\mathrm{lip}}

    \makeatletter
    \let\@fnsymbol\@arabic
    \makeatother

\makeatletter
\def\blfootnote{\xdef\@thefnmark{}\@footnotetext}
\makeatother

\title{Optimal transport, gradient estimates, and pathwise Brownian coupling on spaces with variable Ricci bounds}
\date{\today}
\author{Mathias Braun\footnote{\textsf{braun@iam.uni-bonn.de}.
  Funded by the European Union through the ERC-AdG ``RicciBounds''.}\and
Karen Habermann\footnote{\textsf{habermann@iam.uni-bonn.de}.
  Funded by the Deutsche Forschungsgemeinschaft (DFG, German Research
  Foundation) under Germany's Excellence Strategy - GZ 2047/1,
  Projekt-ID 390685813.}\and
  Karl-Theodor Sturm\footnote{\textsf{sturm@uni-bonn.de}. Funded by the
  Deutsche Forschungsgemeinschaft (DFG, German Research Foundation)
  under Germany's Excellence Strategy - GZ 2047/1, Projekt-ID
  390685813 as well as through the Collaborative Research Center 1060,
  and funded by the European Union through the ERC-AdG ``RicciBounds''.}}
  
\let\originalleft\left			
\let\originalright\right
\renewcommand{\left}{\mathopen{}\mathclose\bgroup\originalleft}
\renewcommand{\right}{\aftergroup\egroup\originalright}

\begin{document}
\blfootnote{University of Bonn, Institute for Applied Mathematics,
  Endenicher Allee 60, 53115 Bonn, Germany.}
\maketitle

\begin{abstract}
Given a metric measure space $(\mms,\met,\meas)$  and a lower semicontinuous, lower bounded function $k\colon \mms\to\R$, we prove the equivalence of the synthetic approaches to \emph{Ricci curvature at $x\in\mms$ being bounded from below by $k(x)$} in terms of
\begin{itemize}
\item the Bakry--Émery estimate $\Delta\Gamma(f)/2 - \Gamma(f,\Delta f) \geq  k\,\Gamma(f)$ in an appropriate weak formulation, and
\item the curvature-dimension condition $\CD(k,\infty)$ in the sense Lott--Sturm--Villani with variable $k$.
\end{itemize}
Moreover, for all $p\in(1,\infty)$, these properties 
hold if and only if the perturbed $p$-transport cost
\begin{equation*}
W_p^{\kk}(\mu_1,\mu_2,t):=\inf_{(\B^1,\B^2)}  \EE\Big[\expo^{\int_0^{2t} p \kk\left(\B^1_{r}, \B^2_{r}\right)/2\d r} \met^p\left(\B^1_{2t},\B^2_{2t} \right)\Big]^{1/p}
\end{equation*}
is nonincreasing in $t$.
The infimum here is taken over pairs of coupled Brownian motions $\B^1$ and $\B^2$ on $\mms$ with given initial distributions $\mu_1$ and $\mu_2$, respectively, and $\smash{\kk(x,y)} := \smash{\inf_\gamma \int_0^1 \k(\gamma_s)\d s}$ denotes the ``average'' of $\k$ along geodesics $\gamma$ connecting $x$ and $y$.

Furthermore, for any pair of initial distributions $\mu_1$ and $\mu_2$ on $X$, we prove the existence of a pair of coupled Brownian motions $\B^1$ and $\B^2$ such that a.s.~for every $s,t\in[0,\infty)$ with $s\leq t$, we have
\begin{equation*}
\met\left(\B_t^1,\B_t^2\right)\le
\expo^{-\int_s^t  \kk\left(\B_r^1,\B_r^2\right)/2\d r} \met\left(\B_s^1,\B_s^2\right).
\end{equation*}
\end{abstract}
        
\tableofcontents

\section{Introduction}
Throughout this paper, the triple $(\mms,\met,\meas)$ is a \emph{metric measure space}, that is, a complete and separable metric space $(\mms,\met)$ equipped with a locally finite measure $\meas$ defined on the Borel $\sigma$-field $\Borel(\mms)$, and $\k\colon\mms\to\R$ is a lower semicontinuous function which is bounded from below. We always assume that $(\mms,\met,\meas)$ is an $\RCD(K,\infty)$ space for some $K\in\R$.

Denote by $\Prob(\mms)$ the space of Borel probability measures on $(\mms,\met)$. For $p\in[1,\infty)$, $\Prob_p(\mms)$ is the set of $\mu\in\Prob(\mms)$ with $\smash{\int_\mms \met^p(x,y)\d\mu(y) < \infty}$ for some $x\in \mms$. As usual, $\W_p$ denotes the $p$-Kantorovich--Wasserstein distance defined through
\begin{equation*}
\W_p(\mu,\nu) := \inf_\pi  \Big(\int_{\mms\times\mms} \met^p(x,y)\d\pi(x,y)\Big)^{1/p},
\end{equation*}
where the infimum is taken over all $\pi\in\Prob(\mms\times\mms)$ with marginals $\mu$ and $\nu$. If it exists, the limit $\vert\dot{\gamma}_t\vert := \lim_{h\to 0} \met(\gamma_{t+h},\gamma_t)/\vert h\vert$ is called \emph{metric speed} of the curve $\gamma\in\Cont([0,1];\mms)$ at $t\in [0,1]$, and we write $\vert \dot{\gamma}\vert$ if $\vert\dot\gamma_t\vert = \vert\dot{\gamma}_s\vert$ for every $s,t\in[0,1]$. Moreover, $\Geo(\mms)$ denotes the space of
 \emph{geodesics} on $\mms$, i.e.~the set of $\gamma\in\Cont([0,1];\mms)$ with $\met(\gamma_t,\gamma_s) = \vert t-s\vert\,\met(\gamma_0,\gamma_1)$ for all $s,t\in [0,1]$.
Similarly, we define $\Geo(\Prob_p(\mms))$ as the space of $\W_p$-geodesics in the space of probability measures. We say that $\boldsymbol{\pi}\in\Prob(\Geo(\mms))$ \emph{represents the $\W_p$-geodesic $(\mu_t)_{t\in [0,1]}$} if $\mu_t = (\eval_t)_\push\boldsymbol{\pi}$ for all $t\in
[0,1]$, where $\eval_t\colon \Cont([0,1];\mms) \to \mms$ is the evaluation map defined by $\eval_t(\gamma) := \gamma_t$.
By \cite{lisini2007}, every $\W_p$-geodesic can be represented by some $\boldsymbol{\pi}\in\Prob(\Geo(\mms))$. 

We present various synthetic approaches to the definition of \emph{Ricci curvature at $x\in\mms$ bounded from below by $\k(x)$} and prove their equivalence. These characterizations are suitable extensions of the curvature-dimension condition, the evolution variational inequality, Bochner's inequality, gradient estimates and transport estimates to nonconstant curvature bounds. To this list, we add a description in terms of pathwise coupling of Brownian motions.  In total, our main result is the following.

\begin{thm}\label{Th:Huge thm} Let $(\mms,\met,\meas)$ be an $\RCD(K,\infty)$ space for some $K\in\R$, and let $\k\colon\mms\to\R$ be a lower semicontinuous, lower bounded function. For all exponents $p\in (1,\infty)$ and $q\in [1,\infty)$, the following properties are equivalent:
\begin{enumerate}[label=\textnormal{(\roman*)}]
\item the curvature-dimension condition $\CD(\k,\infty)$,
\item the evolution variational inequality $\EVI(\k)$,
\item the $q$-Bochner inequality $\BE_q(\k,\infty)$,
\item the $q$-gradient estimate $\GE_q(\k)$,
\item the $p$-transport estimate $\PTE_p(k)$, and
\item the pathwise coupling property $\PCP(k)$.
\end{enumerate}

Moreover, any of these properties yields \textnormal{(iii), (iv)} and \textnormal{(v)} for all exponents $p,q\in [1,\infty)$.
\end{thm}

Let us now introduce each of these extensions and give an overview of the organization of our reasoning. Throughout, we assume the reader to be familiar with the theory of $\RCD(K,\infty)$ spaces and basic properties of these. An account on this will be collected in Section \ref{Ch:Preliminaries} which can be read independently of the rest of this paper. 

\subsection{Lagrangian formulation of synthetic variable Ricci bounds}

Here and in the sequel, $\green(s,t) := \min\{s(1-t),t(1-s)\}$ denotes the Green's function of the unit interval $[0,1]$. Define the \emph{Boltzmann entropy} $\Ent_\meas \colon \Prob(\mms) \to [-\infty,\infty]$ as
\begin{equation*}
\Ent_\meas(\mu) := \int_\mms\rho\log\rho\d\meas\quad \text{if }\mu \ll \meas\text{ with }\mu = \rho\ \!\meas,\quad\Ent_\meas(\mu) := \infty\quad\text{otherwise}.
\end{equation*}
We put $\Dom(\Ent_\meas) := \{\mu\in\Prob(\mms): \Ent_\meas(\mu)\in\R\}$.

\begin{citedefinition}[{\cite[Definition 3.2]{sturm2015}}]\label{CD-def} A metric measure
space $(\mms,\met,\meas)$ is said to satisfy the \emph{curvature-dimension
condition}  with variable curvature bound $\k$, briefly $\CD(\k,\infty)$, if for every
$\mu_0,\mu_1\in\Prob_2(\mms)\cap\Dom(\Ent_\meas)$ there exists a
measure $\boldsymbol{\pi}\in\Prob(\Geo(\mms))$ representing some 
$\W_2$-geodesic $(\mu_t)_{t\in [0,1]}$ connecting $\mu_0$ and $\mu_1$
 such that, for all $t\in
[0,1]$,
\begin{equation*}
\Ent_\meas(\mu_t) \leq (1-t)\Ent_\meas(\mu_0) + t\Ent_\meas(\mu_1) - \int_0^1\int_{\Geo(\mms)} \green(s,t)\,\k(\gamma_s)\,\vert\dot\gamma\vert^2\d\boldsymbol{\pi}(\gamma)\d s.
\end{equation*}
\end{citedefinition}

\begin{citedefinition}[{\cite[Definition 3.3]{sturm2015}}] \label{Def:EVI}
A metric measure
space $(\mms,\met,\meas)$ is said
to satisfy the \emph{evolution variational inequality} with
variable curvature bound $\k$, briefly $\EVI(\k)$, if for
every $\mu_0\in\Prob_2(\mms)$ there exists a locally absolutely continuous 
curve $(\mu_t)_{t>
0}$ in $\Dom(\Ent_\meas)$ with $\W_2(\mu_t,\mu_0) \to 0$ as $t\to 0$,
and for every $t>0$ and $\nu\in\Prob_2(\mms)$ there exists a
measure $\boldsymbol{\pi}_t \in \Prob(\Geo(\mms))$ representing some $\W_2$-geodesic connecting
$\mu_t$ and $\nu$
such that
\begin{equation*}
\frac{\diff^+}{\diff t} \frac{1}{2}\W_2^2(\mu_t,\nu) + \int_0^1\int_{\Geo(\mms)} (1-s)\,\k(\gamma_s)\,\vert\dot\gamma\vert^2\d\boldsymbol{\pi}_t(\gamma)\d s \leq \Ent_\meas(\nu) - \Ent_\meas(\mu_t).
\end{equation*}
\end{citedefinition}

From \cite[Theorem 3.4]{sturm2015}, it is already known that $\CD(\k,\infty)$ is equivalent to $\EVI(\k)$, which establishes the equivalence of (i) and (ii) in Theorem \ref{Th:Huge thm}.

\subsection{Eulerian formulation of synthetic variable Ricci bounds}
Let us now switch to the Eulerian picture which, to shorten the presentation, is directly presented for arbitrary exponents. Define the \emph{Cheeger energy} $\E\colon \Ell^2(\mms,\meas)\to[0,\infty]$ as
\begin{equation*}
\E(f) := \inf\!\Big\lbrace \liminf_{n \to \infty} \int_\mms\lip(f_n)^2\d\meas : f_n \in \Lip_\bdd(\mms),\ \! f_n \to f \text{ in }\Ell^2(\mms,\meas)\Big\rbrace,
\end{equation*}
where $\lip(f)(x) := \limsup_{y\to x} \vert f(x) - f(y)\vert/\!\met(x,y)$ denotes the \emph{local Lipschitz slope} at $x\in\mms$. We put $\Dom(\E) := \smash{\big\lbrace f\in\Ell^2(\mms,\meas) : \E(f) < \infty\big\rbrace}$.

\begin{Def} Given $q\in[1,\infty)$, we say that $(\mms,\met,\meas)$ satisfies the $q$-\emph{Bochner inequality} or $q$-\emph{Bakry--Émery estimate} with variable curvature bound $\k$, briefly  $\BE_q(k,\infty)$, if
\begin{equation*}
\int_\mms \Big(\frac{1}{q}\Gamma(f)^{q/2}\,\Delta\phi - \Gamma(f)^{q/2-1}\,\Gamma(f,\Delta f)\,\phi\Big)\d\meas \geq \int_\mms k\,\Gamma(f)^{q/2}\,\phi\d\meas
\end{equation*}
holds for all $f\in\Dom(\Delta)$ with $\Delta f \in \Dom(\E)$ as well as $\Gamma(f) \in\Ell^\infty(\mms,\meas)$ and for every nonnegative $\phi \in \Dom(\Delta)\cap \Ell^\infty(\mms,\meas)$ with $\Delta\phi\in\Ell^\infty(\mms,\meas)$.
\end{Def}

The equivalence of (i) and (iii) for $q=2$ in our major Theorem \ref{Th:Huge thm} above states that the variable Eulerian and Lagrangian approaches to synthetic lower Ricci bounds coincide, i.e.~$\CD(\k,\infty)$ is equivalent to $\BE_2(\k,\infty)$. If $\k$ is constant, this has been proved by Ambrosio, Gigli and Savaré in their groundbreaking work \cite{ambrosio2015}. In the nonconstant case, this remained open in previous contributions \cite{ketterer2015, ketterer2017, sturm2015}. 

The implication from $\BE_2(\k,\infty)$ to $\CD(\k,\infty)$ follows from Theorem \ref{Th:BEq equiv GEq} and Theorem \ref{Thm:GE2 to CD}.
The proof of the converse is a consequence of Proposition \ref{Pr:EVI implies DTE}, Theorem \ref{thm:nonincreasing}, Theorem \ref{Th: Kuwada TEp to GEq} and eventually Theorem \ref{Th:BEq equiv GEq}. This requires a detailed heat flow analysis, both at the level of functions and measures, 
and in particular an extension of Kuwada's duality \cite[Theorem 2.2]{kuwada2010} between $q$-gradient estimates and $p$-transport estimates for dual $p$ and $q$. This is quite demanding -- indeed, until now not even a formulation of an appropriate $p$-transport estimate with nonconstant curvature bound existed.

The ``self-improvement property'' of the $q$-Bochner inequality will be another key result. Indeed, the $\BE_q(k,\infty)$ condition is \emph{independent} of $q$, see Theorem \ref{Th:q independence}, which provides the equivalence of (i) and (iii) in Theorem \ref{Th:Huge thm} for general $q$.

\subsection{Improved gradient estimates}

Following \cite{sturm2015}, let $\smash{(\Schr^{qk}_t)_{t\geq 0}}$ be the Schrödinger semigroup on $\Ell^2(\mms,\meas)$ associated to the generator $\Delta - qk$ for $q\in [1,\infty)$.
It extends to a strongly continuous semigroup on $\Ell^r(\mms,\meas)$ for each $r\in[1,\infty)$.
In terms of the Brownian motion $({\PP}_x, \B)$ on $\mms$ starting in $x\in\mms$, it can be expressed through the Feynman--Kac formula
\begin{equation}\label{Feynman-Kac}
\Schr^{qk}_t f(x) = \Exp_x\Big[\expo^{-\int_0^{2t}qk(\B_{r})/2\d r}\, f(\B_{2t})\Big]\quad\text{for every }f\in \Ell^r(\mms,\meas).
\end{equation}

\begin{Def} We say that a \emph{$q$-gradient estimate} with variable curvature bound $\k$, briefly $\GE_q(\k)$, holds whenever
\begin{equation*}
\Gamma(\ChHeat_t f)^{q/2} \leq \Schr_t^{qk}\big(\Gamma(f)^{q/2}\big)\quad\meas\text{-a.e.}
\end{equation*}
is satisfied for every $f\in\Dom(\E)$ and every $t\geq 0$.
\end{Def}

Adapting the well-known arguments for constant Ricci curvature bounds from \cite{bakry1985, savare2014}, we establish, as stated in Theorem \ref{Th:BEq equiv GEq}, that $\BE_q(\k,\infty)$ holds if and only if $\GE_q(\k)$ is satisfied. This yields the equivalence of (iii) and (iv) in Theorem \ref{Th:Huge thm} for general $q\in[1,\infty)$.

\subsection{Variable transport estimates}\label{Transport def}

In order to formulate a dual $p$-transport estimate for $p\in[1,\infty)$, we consider evolutions on the product space $\mms\times\mms$. Denoting by $\G_\varepsilon(x,y)$ the set of $\gamma\in\Geo(\mms)$ with $\gamma_0 \in \overline{B}_\varepsilon(x)$ and $\gamma_1\in\overline{B}_\varepsilon(y)$, we introduce the function $\smash{\kk\colon\mms\times\mms\to\R}$ defined by
\begin{equation}\label{k hut}
\kk(x,y) := \lim_{\varepsilon\to 0}\, \inf_{\gamma\in \G_\varepsilon(x,y)}\, \int_0^1 \k(\gamma_s)\d s.
\end{equation}
Its basic properties are summarized in Section \ref{Ch:Preliminaries}. As we will see in Remark \ref{sup improvement}, Theorem \ref{pte to pcp} and Theorem \ref{propn:transport2gradient 1}, it turns out that $\kk$ can indeed equivalently be replaced in all relevant quantities by the larger function $\smash{\kkkkk\colon\mms\times\mms\to\R}$ defined by
\begin{equation}\label{kkkkk def}
\kkkkk(x,y) := \liminf_{(x_n,y_n)\to(x,y)}\, \sup_{\gamma\in\G_0(x_n,y_n)}\,\int_0^1\k(\gamma_s)\d s.
\end{equation}

Given $\mu_1,\mu_2\in\Prob_p(\mms)$, we define the
\emph{perturbed $p$-transport cost} at time $t\geq 0$ by
\begin{equation*}
W_p^{\kk}(\mu_1,\mu_2,t):=\inf_{(\PP, \B^1,\B^2)} \EE\Big[\expo^{\int_0^{2t} p \kk\left(\B^1_{r}, \B^2_{r}\right)/2\d r} \met^p\!\big(\B^1_{2t},\B^2_{2t} \big)\Big]^{1/p},
\end{equation*} 
where the infimum is taken over all pairs of coupled Brownian motions $\smash{\big(\PP,\B^1\big)}$ and $\smash{\big(\PP,\B^2\big)}$ on $\mms$, restricted to $[0,2t]$ and modeled on a common probability space, with initial distributions $\mu_1$ and $\mu_2$, respectively. Note that $\smash{W_p^{\kk}(\mu_1,\mu_2,0)= W_p(\mu_1,\mu_2)}$ and that for general $t\geq 0$, if $\k$ is constant, say $\k = K$, the perturbed $p$-transport cost can be expressed in terms of the usual $p$-transport cost via
\begin{equation*}
\smash{W_p^{\kk}(\mu_1,\mu_2,t)=\expo^{Kt}\,W_p(\WHeat_t\mu_1,\WHeat_t\mu_2)}.
\end{equation*}

\begin{Def} Given any $p\in[1,\infty)$, we say that a \emph{$p$-transport estimate} with variable curvature bound $\k$, briefly $\PTE_p(\k)$, holds if the map $\smash{t\mapsto \W_p^\kk(\mu_1,\mu_2,t)}$ is nonincreasing on $[0,\infty)$ for every pair $\mu_1,\mu_2\in\Prob_p(\mms)$.
\end{Def}

Having at our disposal appropriate replacements for the expressions
$\smash{\expo^{-qKt}\,\ChHeat_t\big(\Gamma(f)^{q/2}\big)}$ and $\smash{\expo^{Kt}\, W_p(\WHeat_t\mu_1,\WHeat_t\mu_2)}$ in terms of Feynman--Kac formulas with potentials $q\k$ for the Brownian motion on $\mms$ and $-p\kk$ for 
pairs of coupled Brownian motions on $\mms\times \mms$, respectively,
we are in a position to formulate and prove a generalization of the fundamental Kuwada duality  in the case of nonconstant $\k$. This addresses the equivalence of (iv) and (v) in Theorem \ref{Th:Huge thm}.

\begin{thm}\label{gra-tra} For every $p,q\in(1,\infty)$ with $1/p+1/q=1$, the following are equivalent:
\begin{enumerate}[label=\textnormal{(\roman*)}]\setcounter{enumi}{3}
\item the $q$-gradient estimate $\GE_q(k)$, and
\item the $p$-transport estimate $\PTE_p(k)$.
\end{enumerate}
\end{thm}

This result is a consequence of Theorem \ref{Th: Kuwada GEq to TEp} and Theorem \ref{Th: Kuwada TEp to GEq}. For both results, it is crucial to use a localization argument in regions where $\k$ or $\kk$ are ``approximately constant'' and then use tail estimates for Brownian paths to control the remainder terms. 

Suitable extensions to the case $q=1$ and $p=\infty$ will be discussed, and eventually shown to be equivalent, in Theorem \ref{Th:Vertical Wasserstein contraction}, Theorem \ref{propn:transport2gradient 1} and Theorem \ref{pte to pcp}. Therefore, making sense of an appropriate $\PTE_p(\k)$ condition for $p= \infty$ is the content of the subsequent Section \ref{Sec:PW coupling}.

\begin{remark}\label{Re:WDTE} It is often convenient to use the characterization of $\PTE_p(\k)$, which is zeroth-order in nature, through a first-order condition via the \emph{differential $p$-transport inequality}
\begin{equation*}
\frac{\diff^+}{\diff t}\bigg\vert_{t=0} \W_p^p(\WHeat_t\delta_x,\WHeat_t\delta_y) \leq - p\,\kk(x,y)\met^p(x,y)\quad\text{for every }x,y\in\mms,
\end{equation*}
very much in the spirit of the connection between $\BE_q(\k,\infty)$ and $\GE_q(\k)$. The equivalence of $\PTE_p(\k)$ and the foregoing estimate, which for constant $\k$ is essentially Gronwall's lemma and a standard coupling technique, is treated in Corollary \ref{Cor:TE ==> WDTE}. 

A posteriori, for every $p\in (1,\infty)$, any of the conditions (i) to (vi) from Theorem \ref{Th:Huge thm} will indeed give the much stronger estimate
\begin{equation*}
\frac{\diff^+}{\diff t}\W_p^p(\WHeat_t\mu_1,\WHeat_t\mu_2) \leq -p\int_0^1\int_{\Geo(\mms)} \k(\gamma_s)\,\vert\dot\gamma\vert^p\d\boldsymbol{\pi}_t(\gamma)\d s\quad\text{for every }t\geq 0,
\end{equation*}
where $\mu_1,\mu_2\in\Prob(\mms)$ have finite $\W_p$-distance to each other, and $\boldsymbol{\pi}_t\in\Prob(\Geo(\mms))$ is an \emph{arbitrary} measure representing a $\W_p$-geodesic from $\WHeat_t\mu_1$ to $\WHeat_t\mu_2$, see Corollary \ref{Cor:Derivative of W_p}.
\end{remark}

\subsection{Pathwise coupling of Brownian motions}\label{Sec:PW coupling}
Finally, we reinforce the $p$-transport estimate by passing to the limit $p\to\infty$ and by replacing the mean value estimates by a pathwise one.

\begin{Def}\label{Def:PCP} We say that the \emph{pathwise coupling property} with variable curvature bound $\k$, briefly $\PCP(\k)$, holds if for every pair $\mu_1,\mu_2\in\Prob(\mms)$ there exists a pair $\smash{\big(\PP,\B^1\big)}$ and $\smash{\big(\PP,\B^2\big)}$ of coupled Brownian motions on $\mms$ with initial distributions $\mu_1$ and $\mu_2$, respectively, such that $\PP$-a.s., we have
\begin{equation*}
\met\!\big(\B_t^1,\B_t^2\big)\le
\expo^{-\int_s^t \kk\left(\B_r^1,\B_r^2\right)/2\d r} \met\!\big(\B_s^1,\B_s^2\big)\quad\text{for every }s,t\in [0,\infty)\text{ with }s\leq t.
\end{equation*}
\end{Def}

It is proved in \cite[Theorem 4.1]{arnaudon2011} that complete Riemannian manifolds with Ricci curvature bounded from below by $K\in\R$ satisfy $\PCP(\k)$ with constant $\k = K$. The work \cite[Theorem 2.9]{sturm2015} extended this to general $\RCD(K,\infty)$ spaces. A first result into the nonconstant direction is due to \cite[Theorem 6]{veysseire2011}. Again on Riemannian manifolds with a uniform lower bound on the Ricci curvature, it deduces the existence of a pair $\smash{\big(\B^1, \B^2\big)}$ of coupled Brownian motions starting in $(x,y)$ obeying for every $t\geq 0$, on the event that $\smash{\big(\B_r^1,\B_r^2\big)}$ does not belong to the cut-locus of $\mms$ for all $r\in [0,t]$, the estimate
\begin{equation*}
\met\!\big(\B_t^1,\B_t^2\big) \leq \expo^{-\int_0^t\kappa\big(\B_r^1,\B_r^2\big)/2\d r}\met(x,y),
\end{equation*}
where $\smash{\kappa(x,y) := - \frac{\diff^+}{\diff t}\big\vert_{t=0}\log\W_1(\WHeat_t\delta_x,\WHeat_t\delta_y)}$ denotes the \emph{coarse curvature} at $x,y\in\mms$, $x\neq y$. For $x,y$ close to each other, say $y=\exp_x(\varepsilon v)$ with $\varepsilon> 0$, $v\in T_x\mms$, we have 
\begin{equation*}
\kappa(x,y) = \Ric_x(v,v)+ \mathrm{o}(1),
\end{equation*}
see \cite[Theorem 19 and Remark 20]{veysseire2011}.
The construction of this process deeply relies on smooth calculus tools, which are unavailable in our setting and thus cannot be adopted. 

Our main theorem extends these results in terms of $\kk$ and circumvents regularity issues involving the variable curvature bound. The existence of a process satisfying the $\PCP(\k)$ condition is even equivalent to $\CD(\k,\infty)$. Indeed, given $\PTE_p(\k)$ for every large enough $p\in (1,\infty)$, we deduce $\PCP(\k)$ by means of Theorem \ref{pte to pcp}, the content of which is the implication from (v) to (vi) in Theorem \ref{Th:Huge thm}. Note that according to the previous Theorem \ref{gra-tra} and nestedness of $q$-gradient estimates, see Lemma \ref{hierarchy}, the $1$-gradient estimate $\GE_1(\k)$ implies $\PTE_p(\k)$ for all $p\in (1,\infty)$ and thus $\PCP(\k)$. The converse of this, i.e.~the implication from $\PCP(\k)$ to $\GE_1(\k)$, is addressed in Theorem \ref{propn:transport2gradient 1}.

\paragraph*{Acknowledgments}  The authors warmly thank Matthias Erbar for a number of fruitful and enlightening discussions.

\section{Preliminaries}\label{Ch:Preliminaries}
\paragraph*{Notations} We write $\Cont(\mms)$ and $\Lip(\mms)$ for the spaces of continuous and Lipschitz functions $f\colon \mms \to \R$, respectively. We set $\Lip(f) := \sup_{x\neq y}\vert f(x) - f(y)\vert/\!\met(x,y)$ for $f\in\Lip(\mms)$. The space of \emph{bounded} continuous functions on $\mms$ is denoted by $\Cont_\bdd(\mms)$, and the space of functions in $\Cont(\mms)$ with \emph{bounded support} is called $\Cont_\bs(\mms)$, and similarly for $\Lip_\bdd(\mms)$ and $\Lip_\bs(\mms)$. 

\paragraph*{The Riemannian curvature-dimension condition} We say that the metric measure space $(\mms,\met,\meas)$ is \emph{infinitesimally Hilbertian} if the Cheeger energy $\E$ is a quadratic form (in other words, if it satisfies the parallelogram identity). Furthermore, we say that $(\mms,\met,\meas)$ satisfies the \emph{Riemannian curvature-dimension condition} $\RCD(\k,\infty)$ if it is infinitesimally Hilbertian and satisfies the curvature-dimension condition $\CD(\k,\infty)$ according to Definition \ref{CD-def}. 
As said, we always assume that $(\mms,\met,\meas)$ is an $\RCD(K,\infty)$ space for some constant $K\in\R$. The value of $K$ does not enter any of our results. Without restriction $\k\geq K$ on $\mms$. Indeed, one should think of $\k$ as being much larger than $K$ everywhere on $\mms$. 

The $\RCD(K,\infty)$ assumption carries numerous important consequences for $(\mms,\met,\meas)$. Further details on the subsequent results can be found in \cite{ambrosio2014a, ambrosio2014b, rajala2014, savare2014}.
\begin{enumerate}[label=\alph*.]
\item{\bf Volume growth.} For each $z\in\mms$ there exists a nonnegative constant $C$ such that $\meas[B_r(z)]\le \expo^{Cr^2}$ for every $r>0$.
\item{\bf Nondegeneracy of entropy.} $\Ent_\meas$ is well-defined and does not attain the value $-\infty$.
\item{\bf Uniqueness of {\boldmath{$W_2$}}-geodesics.} For each pair of $\meas$-absolutely continuous measures $\mu_0,\mu_1\in\Prob_2(\mms)$, there exists a unique $W_2$-geodesic connecting them.
\item{\bf Dirichlet form.} By polarization, $\E$ defines a quasi-regular, strongly local, conservative Dirichlet form, unambiguously denoted by $\E$, on $\Ell^2(\mms,\meas)$ with dense domain $\Sob^{1,2}(\mms):=\Dom(\E)$. The latter is a Hilbert space w.r.t.~$\smash{\big[\Vert f \Vert_{\Ell^2(\mms,\meas)}^2 + \E(f)\big]^{1/2}}$. 
The generator of $\E$, i.e.~the self-adjoint operator $\Delta$ on $\Ell^2(\mms,\meas)$ defined by putting $f\in\Dom(\Delta)$ and $h = \Delta f$ if and only if
\begin{equation*}
\E(f,g) = -\int_\mms h\,g\d\meas\quad\text{for every }g\in\Sob^{1,2}(\mms),
\end{equation*}
is called \emph{Laplacian}.

\item{\bf Heat flow.} The Dirichlet form $\E$ 
defines the heat semigroup $(\ChHeat_t)_{t\ge0}$ as its gradient flow in $\Ell^2(\mms,\meas)$, or alternatively via spectral calculus as $\ChHeat_t=\expo^{\Delta t}$, $t\geq 0$. This semigroup is $\meas$-symmetric and extends to a strongly continuous contraction semigroup on $\Ell^r(\mms,\meas)$ for any $r\in[1,\infty)$. It can be chosen to be strong Feller, more precisely, $\ChHeat_t$ maps $\Ell^\infty(\mms,\meas)$ to $\Lip(\mms)$ for $t>0$ with $\Lip(\ChHeat_tf) \leq \Vert f\Vert_{\Ell^\infty(\mms,\meas)}/\sqrt{t}$ if $K=0$, while if $K\neq 0$, then
\begin{equation}\label{Lip reg est}
\Lip(\ChHeat_tf)^2 \leq \frac{K}{\expo^{2Kt}-1}\, \Vert f\Vert_{\Ell^\infty(\mms,\meas)}^2\quad\text{for every }f\in\Ell^\infty(\mms,\meas).
\end{equation}
The semigroup $(\ChHeat_t)_{t\geq 0}$ is in duality with the semigroup $(\WHeat_t)_{t\ge0}$ defined as the gradient flow of $\Ent_\meas$ in $\Prob_2(\mms)$ and extended  to  $\Prob(\mms)$ by continuity, i.e.
\begin{equation*}
\int_\mms f\d\WHeat_t\mu=\int_\mms \ChHeat_t f\d\mu\quad\text{for every } f\in\Cont_\bdd(\mms)\text{ and }\mu\in\Prob(\mms).
\end{equation*}
In particular, 
$\WHeat_t(g\ \!\meas)=(\ChHeat_tg)\ \! \meas$ for every $g\in\Ell^1(\mms,\meas)$. 

\item{\bf Uniqueness of EVI curves.} Every curve $(\mu_t)_{t\geq 0}$ in $\Prob_2(\mms)$ satisfying the obstructions from Definition \ref{Def:EVI} with arbitrary choice of $\k\geq K$ necessarily coincides with the heat flow $(\WHeat_t\mu_0)_{t \geq  0}$ starting at $\mu_0$.

\item {\bf Brownian motion.} For each $\mu\in\Prob(\mms)$, there exists a conservative Markov process $(\PP,(\B_t)_{t\geq 0})$ on $\mms$, or $(\PP,\B)$ for short, unique in law, with continuous sample paths and transition semigroup given by
\begin{equation*}
\EE\big[f(\B_{t+s}) \mid \B_s\big] = \ChHeat_{t/2}f(\B_s)\quad\text{for every } s,t\in [0,\infty)\text{ and }f\in\Cont_\bdd(\mms),
\end{equation*}
and with 
$(\B_0)_\push\PP=\mu$. This process is called \emph{the} Brownian motion on $\mms$ with initial distribution $\mu$. If we want to stress the dependence on the initial distribution, we write $\PP_\mu$ instead of $\PP$, where we abbreviate $\PP_{\delta_x}$ by $\PP_x$ for $x\in\mms$.

\item{\bf Carré du champ.} The set $\Lip(\mms)\cap \Ell^2(\mms,\meas)$ 
is a core for $\E$. A quadratic functional $\Gamma\colon \Sob^{1,2}(\mms)\to \Ell^1(\mms,\meas)$
can be defined by requiring
\begin{equation*}
\int_\mms\Gamma(f)\,g\d\meas=\E(f,f\,g)-\frac12\E\big(f^2,g\big)\quad\text{for every } g\in \Lip_\bdd(\mms).
\end{equation*}
Indeed, $\Gamma(f)^{1/2}$ coincides $\meas$-a.e.~with the \emph{minimal weak upper gradient} $\vert \Diff f\vert$.

\item{\bf Test functions.} The set 
\begin{equation}\label{Eq:Test functions}
\TestF(\mms) := \big\lbrace f \in\Dom(\Delta) \cap \Ell^\infty(\mms,\meas) : \Gamma(f) \in \Ell^\infty(\mms,\meas),\ \! \Delta f\in\Sob^{1,2}(\mms)\big\rbrace
\end{equation}
is a core  for $\E$ and an algebra w.r.t.~pointwise multiplication. 

\item{\bf Twice differentiability.} We have $\Gamma(f)^{1/2}\in\Dom(\E)$ for all $f\in\D(\Delta)$ and
\begin{equation*}
\E\big(\Gamma(f)^{1/2}\big)\le\|f\|_{\Ell^2(\mms,\meas)}^2-K\|\Delta f\|_{\Ell^2(\mms,\meas)}^2.
\end{equation*}
\item{\bf Sobolev-to-Lipschitz property.} Every $f\in\Sob^{1,2}(\mms,\meas)$ with  $\vert\Diff f\vert\in\Ell^\infty(\mms,\meas)$ has a Lipschitz representative $\overline{f}$ with $\Lip(\overline{f}) \leq \Vert \vert\Diff f\vert\Vert_{\Ell^\infty(\mms,\meas)}$.
\end{enumerate}

\paragraph*{Hopf--Lax evolution} For later use, we summarize the main properties of the general $p$-Hopf--Lax (or Hamilton--Jacobi) semigroup $(Q_s)_{s\geq 0}$, $p\in (1,\infty)$. A detailed account on this topic in general metric spaces can be found in \cite{ambrosio2013, ambrosio2014a, ambrosio2014b}. 

Fix a Lipschitz function $f$ on $\mms$. Its $p$-Hopf--Lax evolution $(Q_sf)_{s\geq 0}$ is defined by
\begin{equation*}
Q_0f := f\quad\text{and}\quad Q_sf(x) := \inf_{y\in\mms} \Big\lbrace f(y) + \frac{\met^p(x,y)}{ps^{p-1}}\Big\rbrace\quad\text{for every }s\in (0,\infty)\text{ and } x\in\mms.
\end{equation*}
The map $s \mapsto Q_sf$ belongs to $\Lip([0,\infty);\Cont(\mms))$, where $\Cont(\mms)$ is endowed with the usual supremum metric. We also have $Q_sf\in \Lip(\mms)$ with $\Lip(Q_sf) \leq p\,\Lip(f)$ for all $s\in (0,\infty)$. Denoting by $q\in (1,\infty)$ the dual exponent to $p$, for every $x\in\mms$, we have
\begin{equation*}\label{Hamilton-Jacobi equ}
\frac{\diff}{\diff s} Q_s f(x) +\frac{1}{q}\lip(Q_sf)^q(x)\leq 0
\end{equation*}
for all but at most countably many $s\in (0,\infty)$, and equality holds e.g.~if $(\mms,\met)$ is geodesic.

Using the $p$-Hopf--Lax semigroup gives a nice duality formula for the $p$-Kantorovich--Wasser\-stein distance, see \cite{kuwada2010, villani2009} for details: for all $\mu,\nu\in\Prob(\mms)$, one has
\begin{equation}\label{Eq:Kantorovich}
\frac{1}{p}\W_p^p(\mu,\nu) = \sup\!\Big\lbrace \int_\mms Q_1f\d\mu - \int_\mms f\d\nu : f\in\Lip_\bdd(\mms)\Big\rbrace.
\end{equation}

\paragraph*{The function {\boldmath{$\kk$}} and Lipschitz approximation} Recall that $\k$ is lower semicontinuous and bounded from below by $K$, and so is $\kk$ by construction. If $\k$ is also bounded from above, say by $C\in\R$, then so is $\kk$. By reparameterization of geodesics, we get $\kk(x,y) = \kk(y,x)$ for every $x,y\in\mms$. Note that $\k$ can be reconstructed from $\kk$, since $\k(x) = \kk(x,x)$. 
Lastly, the function $\kk$ defined in \eqref{k hut} is the pointwise monotone limit from below of bounded Lipschitz functions $\kk_n$, and so is the function $\k$ by considering $\kk_n$ on the diagonal. We intend Lipschitz continuity on $\mms\times\mms$ w.r.t.~the product metric $\met_{\mms\times\mms}$ given by $\smash{\met_{\mms\times\mms}\!\big((x,y),(x',y')\big) := \big[\!\met^2(x,x') + \met^2(y,y')\big]^{1/2}}$. The former fact will be used frequently. Following \cite{ambrosio2008}, we can, for instance, define $\kk_n\colon \mms\times\mms \to \R$ for $n\in\N$ by
\begin{equation*}
\kk_n(x,y) := \inf\!\big\lbrace\! \min\{\kk(x',y'),n\} + n \met_{\mms\times\mms}\!\big((x,y),(x',y')\big) : x',y'\in\mms\big\rbrace.
\end{equation*}

\begin{lma}\label{Le:kk approximation} The above functions $\kk_n$, $n\in\N$, have the following properties:
\begin{enumerate}[label=\textnormal{(\roman*)}]
\item for every $n\in\N$, the function $\kk_n$ is Lipschitz on $\mms\times\mms$ with $\Lip(\kk_n) \leq n$,
\item for all $x\in\mms$ and each $n\in\N$, we have $K \leq \kk_{n}(x)\leq \kk_{n+1}(x) \leq n+1$, and
\item the sequence $(\kk_n)_{n\in\N}$ converges pointwise from below to $\kk$.
\end{enumerate}
\end{lma}

\section[Gradient estimates, Bochner's inequality, and their self-improvements]{Gradient estimates, Bochner's inequality, and their self-im\-provements}\label{Ch:Bochner}
In this section, we adapt the well-known arguments of \cite{bakry1985, savare2014} for constant curvature lower bounds to derive the equivalence of the $q$-Bochner inequality with the $q$-gradient estimate with exponent $q\in[1,\infty)$. Moreover, we prove that these properties are independent of $q$.

Up to replacing $\k$ by $\k_n := \min\{\k,n\}$, $n\in\N$, we may assume throughout this chapter that $\k$ is bounded. In the general case, each of the subsequent results still holds for $\k$ since $\BE_q(\k,\infty)$ and $\GE_q(\k)$ trivially imply $\BE_q(\k_n,\infty)$ and $\GE_q(\k_n)$ for every $n\in\N$, respectively, and conversely, if $\BE_q(\k_n,\infty)$ and $\GE_q(\k_n)$ hold for each $n\in\N$, the monotone convergence theorem implies $\BE_q(\k,\infty)$ and $\GE_q(\k)$, respectively.

\subsection{Equivalence of Bochner and gradient estimate
}\label{Sec:2.2}

First, we review the measure-valued Laplacian $\DELTA$ and the measure-valued $\Gamma_2$-operator $\boldsymbol{\Gamma}_2$ as introduced and analyzed in \cite{gigli2018, savare2014}, defined by means of
\begin{align}\label{DELTA definition}
\int_\mms g\d\DELTA f &= -\int_\mms \Gamma(g,f)\d\meas\quad\text{for every }g\in\Lip_\bs(\mms)\quad\text{and}\\
\boldsymbol{\Gamma}_2(f) &:= \frac{1}{2}\DELTA \Gamma(f) - \Gamma(f,\Delta f)\ \! \meas\nonumber
\end{align}
for suitable $f\in\Sob^{1,2}(\mms)$. We write $f\in\Dom(\DELTA)$ if the signed measure $\DELTA f$ exists, which is then uniquely determined by \eqref{DELTA definition}. We  denote the density of the $\meas$-absolutely continuous part of $\boldsymbol{\Gamma}_2(f)$ by $\gamma_2(f)$. The singular part of $\boldsymbol{\Gamma}_2(f)$ w.r.t.~$\meas$ is a nonnegative measure. 
Both $\DELTA f$ and $\boldsymbol{\Gamma}_2(f)$ are well-defined for $f\in\TestF(\mms)$. Lastly, a well-known consequence of the generic calculus rules of $\Gamma$ proved in \cite{savare2014} is the following chain rule for $\DELTA$.

\begin{lma}\label{Le:Chain rule DELTA} Fix $f\in\Dom(\DELTA)\cap\Ell^\infty(\mms,\meas)$, an interval $I\subset\R$ with $0\in I$ containing the image of $f$, and a function $\Phi\in\Cont^2(I)$ such that $\Phi(0) = 0$. Then $\Phi(f) \in \Dom(\DELTA)$ and
\begin{equation}\label{Eq:DELTA Phi(f)}
\DELTA\Phi(f) = \Phi'(f)\ \!\DELTA f + \Phi''(f)\, \Gamma(f)\ \!\meas.
\end{equation}
\end{lma}

Once $\BE_2(\k,\infty)$ holds, one can argue exactly as for \cite[Lemma 3.2]{savare2014} to get
\begin{align*}
\E\big(\Gamma(f)\big) &\leq -\int_\mms 2\k\, \Gamma(f)^2 + \Gamma(f)\, \Gamma(f,\Delta f)\d\meas\quad\text{and}\\
\k\,\Gamma(f)\ \!\meas &\leq \frac{1}{2}\DELTA\Gamma(f) - \Gamma(f,\Delta f)\ \! \meas.
\end{align*}
for every $f\in\TestF(\mms)$. 
Taking these estimates into account, one can argue exactly as in the proof of \cite[Theorem 3.4]{savare2014} to obtain that, for every $f\in\TestF(\mms)$,
\begin{equation}\label{iterated gamma}
\Gamma(\Gamma(f)) \leq 4 \big(\gamma_2(f) - \k\,\Gamma(f)\big)\,\Gamma(f)\quad\meas\text{-a.e.}
\end{equation}
Using this, we deduce the whole range of $q$-Bochner inequalities from $\BE_2(\k,\infty)$.

\begin{Prop}\label{Th:BE self-improvement} The condition $\BE_2(\k,\infty)$ implies $\BE_q(\k,\infty)$ for every $q\in [1,\infty)$.
\end{Prop}

\begin{proof} Fix $f\in\TestF(\mms)$  and a nonnegative $\phi\in\Dom(\Delta)\cap\Ell^\infty(\mms,\meas)$ with $\Delta\phi\in\Ell^\infty(\mms,\meas)$. Given $\varepsilon > 0$,  consider the smooth function $\smash{\Phi_\varepsilon(r) := (r+\varepsilon)^{q/2} - \varepsilon^{q/2}}$ defined for $r\geq 0$. Since $2-q\leq 1$, we obtain the $\meas$-a.e.~inequalities
\begin{equation*}
-\Gamma(\Gamma(f)) \, \Phi_\varepsilon''(\Gamma(f)) \leq \frac{q}{4}\Gamma(\Gamma(f))\, \big(\Gamma(f) + \varepsilon\big)^{q/2-2} \leq 2 \big(\gamma_2(f) - \k\,\Gamma(f)\big)\, \Phi_\varepsilon'(\Gamma(f))
\end{equation*}
by means of \eqref{iterated gamma}. Multiplying this by $\phi$ and integrating, one gets
\begin{align*}
&-\int_\mms \Gamma(\Gamma(f)) \,\Phi_\varepsilon''(\Gamma(f))\,\phi\d\meas\\
&\qquad\qquad\leq 2\int_\mms  \Phi_\varepsilon'(\Gamma(f))\,\phi\d\boldsymbol{\Gamma}_2(f) - 2\int_\mms  \k\,\Gamma(f)\,\Phi_\varepsilon'(\Gamma(f))\,\phi\d\meas\\
&\qquad\qquad = \int_\mms \Phi_\varepsilon'(\Gamma(f))\,\phi\d\DELTA\Gamma(f) - 2\int_\mms \Phi_\varepsilon'(\Gamma(f))\,\big(\Gamma(f,\Delta f) + \k\,\Gamma(f)\big)\,\phi\d\meas.
\end{align*}
Invoking Lemma \ref{Le:Chain rule DELTA}, this amounts to
\begin{equation*}
2\int_\mms\Phi_\varepsilon'(\Gamma(f))\,\big(\Gamma(f,\Delta f) + \k\,\Gamma(f)\big)\,\phi\d\meas \leq \int_\mms \phi\d\DELTA\Phi_\varepsilon(\Gamma(f)) = \int_\mms\Phi_\varepsilon(\Gamma(f))\,\Delta\phi\d\meas.
\end{equation*}
Note that the left integrand vanishes $\meas$-a.e.~on the set $\{\Gamma(f) = 0\}$ for every $\varepsilon > 0$. Therefore, letting $\varepsilon \downarrow 0$ in the preceding inequality gives the $\BE_q(\k,\infty)$ inequality for $f\in\TestF(\mms)$. 

To extend this to general $f\in\Dom(\Delta)$ with $\Delta f\in \Sob^{1,2}(\mms)$ and $\Gamma(f) \in \Ell^\infty(\mms,\meas)$, we approximate it in $\Sob^{1,2}(\mms)$ by means of its heat flow regularizations $\ChHeat_tf\in\TestF(\mms)$ as $t\downarrow 0$. Since $\Gamma(\ChHeat_tf) \to \Gamma(f)$ and $\Gamma(\ChHeat_tf,\Delta\ChHeat_tf) \to \Gamma(f,\Delta f)$ in $\Ell^1(\mms,\meas)$ as $t\downarrow 0$, $\Gamma(\ChHeat_tf)$ is uniformly bounded in $\Ell^\infty(\mms,\meas)$ for small enough $t$, and $\Gamma(\Delta\ChHeat_tf)^{1/2}$ is uniformly bounded in $\Ell^2(\mms,\meas)$ for small enough $t$, we easily get
\begin{equation*}
\lim_{t\downarrow 0} \Gamma(\ChHeat_tf)^{q/2} = \Gamma(f)^{q/2}\quad\text{and}\quad \lim_{t\downarrow 0}\Gamma(\ChHeat_t f)^{q/2-1}\, \Gamma(\ChHeat_tf,\Delta\ChHeat_tf) = \Gamma(f)^{q/2-1}\,\Gamma(f,\Delta f)
\end{equation*}
in $\Ell^1(\mms,\meas)$. This yields the claim.
\end{proof}

By the Feynman--Kac representation \eqref{Feynman-Kac} of $\smash{\Schr_t^{q\k}}$ and Jensen's inequality, the following hierarchy between gradient estimates is immediate. This and the above self-improvement property of $\BE_2(\k,\infty)$ will be used in the proof of Theorem \ref{Th:BEq equiv GEq} below.

\begin{lma}\label{hierarchy} If $\GE_q(\k)$ holds for some $q\in [1,\infty)$, then $\GE_{q'}(\k)$ is satisfied for all $q'\in [q,\infty)$.
\end{lma}

\begin{thm}\label{Th:BEq equiv GEq} For every $q\in [1,\infty)$, the properties  $\BE_q(\k,\infty)$ and  $\GE_q(\k)$ are equivalent to each other.
\end{thm}
\begin{proof}
By density of $\TestF(\mms)$ in $\Sob^{1,2}(\mms)$ and an argument as in the proof of Proposition \ref{Th:BE self-improvement}, the function $f$ under consideration may be assumed to belong to $\TestF(\mms)$.

Suppose that $\BE_q(\k,\infty)$ is satisfied. Fix any $t>0$, $f$ as above and a nonnegative $\phi\in\Dom(\Delta) \cap \Ell^\infty(\mms,\meas)$ with $\Delta\phi\in \Ell^\infty(\mms,\meas)$. 
Given any $\varepsilon > 0$, consider the function $\Phi_\varepsilon$ as defined in the proof of Proposition \ref{Th:BE self-improvement} above. Define $F_\varepsilon\colon [0,t]\to \R$ by
\begin{equation*}
F_\varepsilon(s) := \int_\mms \Schr_s^{q\k}\big(\Phi_\varepsilon\big(\Gamma(\ChHeat_{t-s}f)\big)\big)\,\phi\d\meas =  \int_\mms \Phi_\varepsilon\big(\Gamma(\ChHeat_{t-s}f)\big)\, \Schr_s^{q\k}\phi\d\meas.
\end{equation*}
This function belongs to $\Cont^1([0,t])$ since the functions $s\mapsto \smash{\Schr_s^{qk}\phi}$ and $s\mapsto \Phi_\varepsilon\big(\Gamma(\ChHeat_{t-s}f)\big)$ as well as their derivatives in $\Ell^2(\mms,\meas)$ are bounded on $[0,t]$, see also \cite[Lemma 2.1]{ambrosio2015} for a similar argument. Thus
\begin{align*}
\liminf_{\varepsilon \downarrow 0} F_\varepsilon'(s) &\geq \liminf_{\varepsilon\downarrow 0}\int_\mms \Phi_\varepsilon\big(\Gamma(\ChHeat_{t-s}f)\big)\,(\Delta-q\k)\Schr_s^{q\k}\phi\d\meas\\
&\qquad\qquad - 2\,\limsup_{\varepsilon \downarrow 0}\int_\mms \Phi_\varepsilon'\big(\Gamma(\ChHeat_{t-s}f)\big)\, \Gamma(\ChHeat_{t-s}f,\Delta\ChHeat_{t-s}f)\,\Schr_s^{q\k}\phi\d\meas,
\end{align*}
which is nonnegative by $\BE_q(\k,\infty)$. Fatou's lemma gives
\begin{equation*}
F_0(t) - F_0(0) = \liminf_{\varepsilon \downarrow 0} \big(F_\varepsilon(t) - F_\varepsilon(0)\big) \geq \int_0^t\liminf_{\varepsilon\downarrow 0} F_\varepsilon'(s)\d s \geq 0,
\end{equation*}
which establishes $\GE_q(\k)$ for $f\in\TestF(\mms)$ by the arbitrariness of $\phi$.

Conversely, assume $\GE_q(\k)$ for $q\in[2,\infty)$. As $\Phi_0\in \Cont^1([0,\infty))$ for such $q$, we deduce $F_0'(0) \geq 0$, which is a reformulation of the $\BE_q(\k,\infty)$ inequality with $\ChHeat_t f$ in place of $f$. Letting $t\downarrow 0$ gives the desired conclusion. If, on the other hand, we have $q \in [1,2)$, we cannot rely on the above regularity of $\Phi_0$. However, Lemma \ref{hierarchy} ensures $\GE_2(\k)$, which implies $\BE_2(\k,\infty)$ by the previous discussion. Therefore, $\BE_q(\k,\infty)$ holds by Proposition \ref{Th:BE self-improvement}.
\end{proof}

\subsection[Independence of the $q$-Bochner inequality on $q$]{Independence of the {\boldmath{$q$}}-Bochner inequality on {\boldmath{$q$}}}

\begin{thm}\label{Th:q independence} If the $q$-Bakry--Émery estimate $\BE_q(\k,\infty)$ holds for some $q\in[1,\infty)$, then it holds for every $q\in[1,\infty)$.
\end{thm}

Lemma \ref{hierarchy} and Proposition \ref{Th:BE self-improvement} give the assertion of this theorem when starting with $\BE_q(\k,\infty)$ for $q\in [1,2]$. To cover the range $q\in (2,\infty)$, we adapt the arguments of \cite{han2018} to prove that a $q$-Bakry--Émery inequality $\BE_q(\k,\infty)$ for some $q\in [1,\infty)$ implies $\BE_2(\k,\infty)$. A crucial point in this argument is that our a priori assumption $\RCD(K,\infty)$ guarantees  $\Gamma(f)^{q/2} \in \Dom(\DELTA)$ for all $f\in\TestF(\mms)$ and \emph{every} $q\in[1,\infty)$.

Arguing exactly as in the constant situation in \cite[Lemma 3.2]{han2018} (see also \cite[Theorem 3.4]{savare2014}), one can show that $\BE_q(\k,\infty)$ holds if and only if the inequalities
\begin{align}\label{Eq:Observation}
\frac{1}{2} \Gamma(f)\, \delta(\Gamma(f)) + \frac{q-2}{4}\Gamma(\Gamma(f)) \geq \Gamma(f)\,\Gamma(f,\Delta f) + \k\,\Gamma(f)^2\quad\meas\text{-a.e.}\quad\text{and}\quad \overline{\Gamma(f)}\,\DELTA_{\perp}\Gamma(f) \geq 0
\end{align}
are valid for every $f\in\TestF(\mms)$. Here, $\delta(\Gamma(f))$ denotes the density of the $\meas$-absolutely continuous part of $\DELTA\Gamma(f)$ w.r.t.~$\meas$, $\DELTA_{\perp}\Gamma(f)$ stands for the corresponding $\meas$-singular part, and $\overline{\Gamma(f)}$ is the quasi-continuous representative of $\Gamma(f)$.
\begin{Prop} Let $\BE_q(\k,\infty)$ be satisfied for some $q\in[1,\infty)$. Then $\BE_2(\k,\infty)$ holds.
\end{Prop}

\begin{proof} As discussed above, it suffices to show the claimed implication starting from $\GE_q(\k)$ with $q\in (2,\infty)$. 
Due to our standing assumption $\RCD(K,\infty)$, the set $\TestF(\mms)$ is dense in $\W^{1,2}(\mms)$, thus it is enough to check the $\BE_2(\k,\infty)$ inequality for $f\in\TestF(\mms)$. Moreover, note that $\GE_q(\k)$ already yields $\overline{\Gamma(f)}\,\DELTA_{\perp}\Gamma(f) \geq 0$ by \eqref{Eq:Observation} which is independent of $q$.

The crucial point is to show that
\begin{equation}\label{Eq:Epsilon inequality}
\frac{1}{2} \Gamma(f)\,\delta(\Gamma(f)) + \varepsilon\Gamma(\Gamma(f)) \geq \Gamma(f)\,\Gamma(f,\Delta f) + \k\,\Gamma(f)^2\quad\meas\text{-a.e.}
\end{equation}
for every $\varepsilon > 0$. Given the observation \eqref{Eq:Observation}, this will imply $\BE_{2+4\varepsilon}(\k,\infty)$ for each $\varepsilon > 0$, and eventually letting $\varepsilon \downarrow 0$ and applying the monotone convergence theorem, we get the claimed $\BE_2(\k,\infty)$ condition.

Given $\BE_{q'}(\k,\infty)$ for arbitrary $q'\geq q$, it is straightforward to follow the proof of \cite[Theorem 3.6]{han2018}, which relies on generic calculus rules for $\boldsymbol{\Gamma}_2$ and closely follows the strategy presented in \cite{savare2014}, to prove \eqref{Eq:Epsilon inequality} with $\varepsilon$ replaced by $\smash{q' - \frac{1}{4(q'+1)}}$. Now, according to \cite[Lemma 3.3]{han2018}, given any $\varepsilon > 0$ there exist $n\in\N$ and $q' \geq q$ so that $P^n(q') = \varepsilon$, where $P(r) := \smash{r - \frac{1}{4(r+1)}}$ and $P^n$ is the $n$-fold composition of $P$. Since $\BE_q(\k,\infty)$ yields $\BE_{q'}(\k,\infty)$, iterating the foregoing reasoning allows us to finally reach the inequality \eqref{Eq:Epsilon inequality}.
\end{proof}

As for \cite[Proposition 3.7]{han2018}, it is possible to obtain an equivalent characterization of $\BE_2(\k,\infty)$ in terms of a lower bound on the measure-valued \emph{Ricci tensor}
\begin{equation*}
\boldsymbol{\mathrm{Ric}}(\nabla f,\nabla f) := \boldsymbol{\Gamma}_2(f) - \big\vert \!\Hess f\big\vert_{\mathsf{HS}}^2\ \!\meas\quad\text{for every }f\in\TestF(\mms)
\end{equation*}
introduced in \cite{gigli2018}. As for the measure-valued Laplacian $\DELTA$, we denote by $\mathrm{ric}(\nabla f,\nabla f)$ the density of the $\meas$-absolutely continuous part and by $\boldsymbol{\mathrm{Ric}}_\perp(\nabla f,\nabla f)$ the $\meas$-singular part of $\boldsymbol{\mathrm{Ric}}(\nabla f,\nabla f)$, respectively.

\begin{cor} The metric measure space $(\mms,\met,\meas)$ satisfies $\BE_2(\k,\infty)$ if and only if for every $f\in\TestF(\mms)$, we have
\begin{equation*}
\mathrm{ric}(\nabla f,\nabla f) \geq \k\,\Gamma(f)\quad\meas\text{-a.e.}\quad\text{and}\quad\boldsymbol{\mathrm{Ric}}_\perp(\nabla f,\nabla f) \geq 0.
\end{equation*}
\end{cor}

\subsection{Localization of Bochner's inequality}

To study a suitable local-to-global behavior of the $q$-Bochner inequality, we present a reformulation of it where we enlarge the class of functions $\phi$. Recall that our standing assumption $\RCD(K,\infty)$ implies  $\Gamma(f)^{q/2} \in\Sob^{1,2}(\mms)$ for every $f\in\TestF(\mms)$ and $q\in[1,\infty)$.

\begin{lma}\label{First order BE} Given $q\in[1,\infty)$, the $\BE_q(\k,\infty)$ property holds if and only if for all $f\in\TestF(\mms)$ and all nonnegative $\phi\in\Sob^{1,2}(\mms)\cap\Ell^\infty(\mms,\meas)$,
\begin{equation}\label{Gamma BE}
-\int_\mms \Big(\frac{1}{q}\Gamma\big(\Gamma(f)^{q/2},\phi\big) + \Gamma(f)^{q/2-1}\,\Gamma(f,\Delta f)\,\phi\Big)\d\meas \geq \int_\mms k\,\Gamma(f)^{q/2}\,\phi\d\meas.
\end{equation}
\end{lma}

\begin{proof} Obtaining $\BE_q(\k,\infty)$ from \eqref{Gamma BE} through integration by parts and the density of $\TestF(\mms)$ in $\Sob^{1,2}(\mms)$ is easy, thus we focus on the converse. Trivially, the inequality \eqref{Gamma BE} holds for all $\phi\in\Dom(\Delta)\cap\Ell^\infty(\mms,\meas)$ with $\Delta\phi\in\Ell^\infty(\mms,\meas)$. Recall now, e.g.~from \cite{gigli2018, savare2014}, that any function $\phi\in\Sob^{1,2}(\mms)\cap\Ell^\infty(\mms,\meas)$ can be approximated in $\Sob^{1,2}(\mms)$ by means of a mollified heat flow
\begin{equation*}
\mathfrak{P}_\varepsilon \phi := \int_0^\infty \eta(s)\, \ChHeat_{\varepsilon s}\phi\d s,\quad\text{where}\quad\eta\in\Cont_{\mathrm{c}}^\infty((0,\infty);[0,\infty))\quad\text{with}\quad\int_0^\infty\eta(s)\d s = 1,
\end{equation*}
as $\varepsilon \downarrow 0$. Since $\mathfrak{P}_\varepsilon\phi \in\Dom(\Delta)\cap\Ell^\infty(\mms,\meas)$ and $\Delta\mathfrak{P}_\varepsilon\phi = -\int_0^\infty \eta'(s)\, \ChHeat_{\varepsilon s}\phi\d s/\varepsilon\in\Ell^\infty(\mms,\meas)$ for every $\varepsilon > 0$, this allows us to extend the class of admissible $\phi$.
\end{proof}

\begin{Def}\label{Def:Local BE} We say that the \emph{local $q$-Bakry--Émery condition} with variable curvature bound $\k$, briefly $\BE_{q,\loc}(\k,\infty)$, with $q\in[1,\infty)$ holds if for every $z\in \mms$ there exists $\delta > 0$ such that 
\begin{equation*}
-\int_\mms \Big(\frac{1}{q}\Gamma\big(\Gamma(f)^{q/2},\phi\big) + \Gamma(f)^{q/2-1}\,\Gamma(f,\Delta f)\,\phi\Big)\d \meas\geq \int_\mms\k\,\Gamma(f)^{q/2}\,\phi\d\meas
\end{equation*}
for all $f\in\TestF(\mms)$ and every nonnegative $\phi\in\Sob^{1,2}(\mms)\cap\Ell^\infty(\mms,\meas)$ with $\supp\phi \subset B_\delta(z)$.  
\end{Def}

It is elementary to pass from the global $\BE_q(\k,\infty)$ condition to $\BE_{q,\loc}(\k,\infty)$. The converse is more involved.

\begin{thm}\label{Thm:BE local global} For $q\in[1,\infty)$, the property $\BE_{q,\loc}(\k,\infty)$ implies the $\BE_q(\k,\infty)$ condition.
\end{thm}

\begin{proof} Let $\{z_i: i\in \N\}$ be a countable dense subset of $\mms$ and consider the collection of metric balls $\smash{B_{\delta_i}(z_i)}$ with $\delta_i>0$ chosen in such a way that the local $q$-Bakry--Émery inequality is satisfied around $z_i$. 
For $i\in\N$, define functions on $\mms$ by
\begin{equation*}
\eta^0_i := \frac2{\delta_i}\!\met\!\big(\cdot,X\setminus B_{\delta_i}(z_i)\big), \quad \eta^*_i:= \min\!\Big\lbrace \sum_{j=1}^{i}\eta_j^0, 1\Big\rbrace \quad\text{and}\quad \eta_i:=\eta_i^*-\eta_{i-1}^*.
\end{equation*}
Then $\eta_i\in\Lip_\bdd(X)$ with support in $B_{\delta_i}(z_i)$ and 
$\sum_{i=1}^\infty\eta_i=1$ on $\mms$.
Thus, for arbitrary nonnegative $\phi\in\Sob^{1,2}(\mms)\cap\Ell^\infty(\mms,\meas)$, the assumption $\BE_{q,\loc}(\k,\infty)$ allows us to deduce
\begin{align*}
&-\int_\mms \Big(\frac{1}{q}\Gamma\big(\Gamma(f)^{q/2},\phi\big) + \Gamma(f)^{q/2-1}\,\Gamma(f,\Delta f)\,\phi\Big)\d \meas\\
&\qquad\qquad =-\sum_{i=1}^\infty\int_\mms \Big(\frac{1}{q}\Gamma\big(\Gamma(f)^{q/2},\phi\,\eta_i\big) + \Gamma(f)^{q/2-1}\,\Gamma(f,\Delta f)\,\phi\,\eta_i\Big)\d \meas\\
&\qquad\qquad \geq\sum_{i=1}^\infty \int_\mms\k\,\Gamma(f)^{q/2}\,\phi\,\eta_i\d\meas  = \int_\mms\k\,\Gamma(f)^{q/2}\,\phi\d\meas.
\end{align*}
We conclude the assertion using Lemma \ref{First order BE} above.
\end{proof}

\section{From \boldmath{$2$}-gradient estimates to CD and differential \boldmath{$2$}-trans\-port estimates}\label{Ch:Grad to CD and DTE}
Our goal now is to derive the evolution variational inequality $\EVI(\k)$ with variable curvature bound $\k$ from the $2$-gradient estimate $\GE_2(\k)$.  
In \cite{sturm2015} there is a first part of the  proof for this implication. With some extra arguments, we  complete it. 

The key point is a localization argument. Indeed, it suffices to prove the $\EVI(\k)$ ``locally'', that is, for measures in a given small neighborhood. The heat flow will neither stay within this neighborhood nor in any other bounded region. We thus modify it by truncating its tails. Due to the Gaussian behavior of the heat flow, the difference is of arbitrary polynomial order for small times. This will imply the $\CD(\k,\infty)$ inequality locally. However, the latter is already known to give the $\CD(\k,\infty)$ inequality globally, and this in turn yields the global version of the $\EVI(\k)$.

\subsection{Tail estimates for the heat flow}

Given any ball $B_{\delta}(z)\subset \mms$ with $\delta > 0$ and $z\in\mms$, and $\rho\in\Prob(\mms)$, we put
\begin{equation*}
\WHeat_t^*\rho := \One_{B_{2\delta}(z)}\ \!\WHeat_t\rho + \WHeat_t\rho[X\setminus B_{2\delta}(z)]\ \!\delta_z.
\end{equation*}

\begin{lma}\label{lma-trunc} Assume that $\rho \in \Prob(\mms)$ is $\meas$-absolutely continuous with density $f\in \Ell^2(\mms,\meas)$ and $\supp\rho\subset B_\delta(z)$. Then for every $a>0$ there exists $t_*>0$ such that for all $t \in [0,t_*]$ and all bounded Borel functions $\phi$, we have
\begin{equation*}
\W_2^2(\WHeat_t^*\rho,\WHeat_t\rho) \le t^a\quad\text{and}\quad
\Big\vert\!\int_\mms \phi\d\WHeat_t^*\rho - \int_\mms \phi \d\WHeat_t\rho\Big\vert \le t^a\, \sup\vert\phi\vert(\mms).
\end{equation*}
\end{lma}

\begin{proof}
To see the first assertion for $t>0$, the case $t=0$ being trivial, observe that
\begin{align*}
\W_2^2(\WHeat_t^*\rho,\WHeat_t\rho) &\leq \int_{\mms\setminus B_{2\delta}(z)} \met^2(z,x)\d\WHeat_t\rho(x)\\
&\leq \sum_{n=3}^\infty (n\delta)^2\int_{B_{n\delta}(z)\setminus B_{(n-1)\delta}(z)} \ChHeat_t f\d\meas\\
&\le \Vert f\Vert_{\Ell^2(\mms,\meas)}
 \, \sum_{n =3}^\infty (n\delta)^2 \,\Big(\meas\big[B_{n\delta}(z) \setminus B_{(n-1)\delta}(z)\big]\Big)^{1/2}\, \expo^{-(n-2)^2\delta^2/4t}
\end{align*}
where the last inequality comes from the integrated Gaussian heat kernel estimate of \cite[Lemma 1.7]{sturm1995}. Therefore, by the volume growth property in $\RCD(K,\infty)$ spaces and finally assuming that $t$ is small enough, we obtain
\begin{align*}
\W_2^2(\WHeat_t^*\rho,\WHeat_t\rho) &\le \Vert f\Vert_{\Ell^2(\mms,\meas)}\, \Big(\sum_{n=3}^\infty \meas\big[B_{n\delta}(z)\setminus B_{(n-1)\delta}(z)\big]\, \expo^{-n^2\delta^2/72t}\Big)^{1/2}\, \expo^{-\delta^2/8t}\\
&\le \Vert f\Vert_{\Ell^2(\mms,\meas)}\,\Big(\int_\mms\expo^{-\met^2(z,x)/72t}\d\meas(x)\Big)^{1/2}\, \expo^{-\delta^2/8t}\le t^a\!.
\end{align*}

The second assertion follows from the first one, since 
\begin{align*}
\Big\vert\!\int_\mms \phi \d\WHeat_t^*\rho - \int_\mms\phi\d\WHeat_t\rho\Big\vert \le \sup\vert\phi\vert(\mms)\, \WHeat_t\rho[\mms\setminus B_{2\delta}(z)] \le  \frac{\sup\vert\phi\vert(\mms)}{\delta^2}\, \W_2^2(\WHeat_t^*\rho, \WHeat_t\rho).\tag*{\qedhere}
\end{align*}
\end{proof}

In Chapter \ref{sec:transport}, we need the following result, which is a consequence of  Lemma \ref{lma-trunc}. 

\begin{lma}\label{tail estimate} For each $z\in X$, $\delta>0$ and $a>0$ there exists $t_*>0$ such that
\begin{equation*}
\PP_x\big[\B_t^x\notin B_{3\delta}(z)\big] \le t^a\quad\text{for every }x\in B_\delta(z) \text{ and } t\in [0,t_*],
\end{equation*}
where $\smash{\big(\PP_x,\B^x\big)}$ denotes Brownian motion on $\mms$ starting in $x$.
\end{lma}

\begin{proof} 
Let $\rho$ be the uniform distribution of $B_{\delta/2}(z)$. Choose a pair $\big(\PP,\B^x)$ and $(\PP,\B)$ of coupled Brownian motions with initial distributions $\delta_x$ and $\rho$, respectively, such that  
$\smash{\met\left(\B_t^x,\B_t\right)} \leq \smash{\expo^{-Kt}\,\met(x,\B_0)}$ $\PP$-a.s.~for every $t\geq 0$, see \cite[Theorem 2.9]{sturm2015} for the construction. Thus in particular, $\PP$-a.s.~we have
$$\smash{\met\left(\B_t^x,\B_t\right)} \leq \delta$$
for every $t\in [0,t_*']$ and a suitable $t_*' > 0$.
According to the previous Lemma \ref{lma-trunc},
\begin{equation*}
\PP\big[\B_t \notin B_{2\delta}(z)\big]
\le t^a
\end{equation*}
for all $t\in [0,t_*]$ and some $t_* > 0$ depending only on $\meas[B_{\delta/2}(z)]$ and $a$. Combining both estimates yields that
\begin{equation*}
\PP\big[\B_t^x\notin B_{3\delta}(z)\big]
\le\PP\big[\B_t \notin B_{2\delta}(z)\big]
\leq t^a.
\end{equation*}
uniformly in $x\in B_\delta(z)$ for small enough times.
\end{proof}

\subsection[From $2$-gradient estimate to  $\CD$]{From \boldmath{$2$}-gradient estimates to  \boldmath{$\CD$}}\label{Sec:3.1}

In this section, we assume that $\k$ is Lipschitz and bounded. The general case follows using the approximation scheme via the sequence $(\k_n)_{n\in\N}$ with $\k_n(x) := \kk_n(x,x)$ for $x\in\mms$ derived from Lemma \ref{Le:kk approximation}. Indeed, $\GE_2(\k)$ trivially implies $\GE_2(\k_n)$ for every $n\in\N$, which will imply both $\CD(\k_n,\infty)$ and $\EVI(\k_n)$. Since $\W_2$-geodesics between $\meas$-absolutely continuous measures and $\EVI(\k)$-curves are unique, we may then pass to the limit $n\to\infty$ by monotone convergence.

We present a modification of \cite[Lemma 3.5]{sturm2015} which is proved in exactly the same way as the previous version subject to the choice of parameterization from \cite[Theorem 4.16]{ambrosio2015} involving the additional parameter $\kappa$. Throughout this section, we denote by $(Q_s)_{s\geq 0}$ the $2$-Hopf--Lax semigroup.

\begin{lma}\label{improved old lemma} Assume the $2$-gradient estimate $\GE_2(\k)$
with variable curvature bound $\k$, and let $\kappa\in\R$ be an arbitrary constant.
Let $(\rho_s)_{s\in[0,1]}$ with $\rho_s = f_s\,\meas$ be a regular curve in the sense of \cite[Definition 4.10]{ambrosio2015}, and for $t>0$, define
$\smash{\vartheta_{\kappa,t}(s) := \frac{\expo^{\kappa st}-1}{\expo^{\kappa t}-1}}$
if $\kappa \neq 0$ and $\vartheta_{0,t}(s) := s$ as well as $\mathrm{R}_\kappa(t) := \frac{\kappa t}{\expo^{\kappa t} - 1}$ if $\kappa \neq 0$ and $\mathrm{R}_0(t) := 1$. 
Then
\begin{align*}
&\int_\mms Q_1\phi\d\WHeat_t\rho_1 -\int_\mms\phi\d\rho_0 - \frac{1}{2}\mathrm{R}_\kappa^2(t) \int_0^1\big\vert\dot{\rho}_{\vartheta_{\kappa,t}(s)}\big\vert^2\d s + t\, \big(\!\Ent_\meas(\WHeat_t\rho_1)-\Ent_\meas(\rho_0)\big)\nonumber\\
&\qquad\qquad \leq -\int_0^1\int_0^{st}\int_X \ChHeat_r\left((k-\kappa)\,\Schr_{st-r}^{2(\k-\kappa)}\Gamma(Q_s\phi)\right)\d\rho_{\vartheta_{\kappa,t}(s)}\d r\d s
\end{align*}
is satisfied for every $\phi\in\Lip_\bs(\mms)$  and all $t>0$. The term $\big\vert\dot{\rho}_{\vartheta_t(s)}\big\vert$ has to be understood as the metric speed of the original curve $(\rho_s)_{s\in[0,1]}$ evaluated at $\vartheta_t(s)$.

The same estimate is satisfied for every $\W_2$-geodesic $(\rho_s)_{s\in[0,1]}$ with $\meas$-absolutely continuous measures, in which case $\smash{\int_0^1\vert\dot{\rho}_{\vartheta_{\kappa,t}(s)}\vert^2\d s=\W_2^2(\rho_0,\rho_1)}$, independently of $\kappa$ and $t$.
\end{lma}

\begin{lma}
\label{local evi}
Assume the $2$-gradient estimate 
$\GE_2(k)$
with variable curvature bound $k$. Suppose that $k\ge K_z$ in $B_{2\delta}(z)$ for some $z\in\mms$, $K_z\in\R$ and $\delta> 0$. Then for all $\rho_0,\rho_1\in\Prob_2(X)\cap\Dom(\Ent_\meas)$ with support in $B_\delta(z)$ and bounded densities w.r.t.~$\meas$, we have
\begin{equation*} 
\frac{\diff^+}{\diff t}\bigg\vert_{t=0} \frac12 \W_2^2(\WHeat_t\rho_1,\rho_0) + \frac{K_z}{2}\W_2^2(\rho_0,\rho_1) \leq \Ent_\meas(\rho_0) - \Ent_\meas(\rho_1).
\end{equation*}
\end{lma}

\begin{proof} The proof follows the reasoning for
\cite[Lemma~3.6]{sturm2015} and \cite[Theorem~4.16]{ambrosio2015}, but with a subtle modification. Fix $t>0$. While the curve $\smash{(\WHeat_{ts}\rho_{\vartheta_t(s)})_{s\in[0,1]}}$ connects $\rho_0$ and $\WHeat_t\rho_1$, the potentials $Q_s\phi_t$, $s\in[0,1]$, are Hopf--Lax interpolations of optimal Kantorovich potentials for the transport from $\rho_0$ to $\WHeat_t^*\rho_1$. Thus, we have to match these two different situations and then use the nice behavior of the remainder terms.

We know by \cite[Proposition 3.9]{ambrosio2014a} that for any $\W_2$-optimal coupling $\pi_t \in \Prob(\mms\times\mms)$ of $\rho_0$ and $\WHeat_t^*\rho_1$, and any Kantorovich potential $\varphi_t$ relative to $\pi_t$, we have $\vert\Diff\varphi_t\vert \leq \met(x,y) \leq 4\delta$ for $\pi_t$-a.e.~$(x,y)\in\mms\times\mms$. Taking \eqref{Eq:Kantorovich} and the bounded support of $\rho_0$ into account, 
\begin{equation*}
\frac{1}{2}\W_2^2(\WHeat_t^*\rho_1,\rho_0) = \sup\!\Big\lbrace \int_\mms Q_1f\d\WHeat_t^*\rho_1 - \int_\mms f\d\rho_0 : f\in\Lip_\bs(\mms),\ \!\Lip(f) \leq 4\delta\Big\rbrace.
\end{equation*}
The latter supremum is attained, see \cite[Proposition 2.12]{ambrosio2014a}, at some $\phi_t \in \Lip_\bs(\mms)$. Possibly adding constants and invoking a cutoff argument, we may assume that $|\phi_t|\le C$ everywhere on $\mms$ for some $C>0$ independent of $t$. Thus, $\vert Q_s\phi_t\vert$ is bounded on $\mms$ and $\Lip(Q_s\phi_t) \leq 8\delta$, uniformly in $s\in [0,1]$.

Let $(\rho_s)_{s\in [0,1]}$ be the $\W_2$-geodesic joining $\rho_0$ and $\rho_1$. Note that the measures $\rho_s = f_s\ \!\meas$, $s\in [0,1]$, are supported in $B_{2\delta}(z)$. The $\CD(K,\infty)$ condition furthermore ensures that the $f_s$ are bounded uniformly in $s$. Applying Lemma \ref{improved old lemma} with $\kappa:=K_z$ we get
\begin{align*}
&\frac{1}{2t}\Big(\W_2^2(\WHeat_t\rho_1,\rho_0) - \W_2^2(\rho_0,\rho_1)\Big)\\
&\qquad\qquad =  \frac{1}{2t}\Big(\W_2^2(\WHeat_t\rho_1,\rho_0) - \W_2^2(\WHeat_t^*\rho_1,\rho_0) + 2\int_\mms Q_1\phi_t\d\WHeat_t^*\rho_1- 2\int_\mms \phi_t\d\rho_0 - \W_2^2(\rho_0,\rho_1)\Big)\\
&\qquad\qquad \leq \frac{1}{2t}\Big(\W_2^2(\WHeat_t\rho_1,\rho_0) - \W_2^2(\WHeat_t^*\rho_1,\rho_0) + 2\int_\mms Q_1\phi_t\d\WHeat_t^*\rho_1 - 2\int_\mms Q_1\phi_t\d\WHeat_t\rho_1\Big) \\
&\qquad\qquad\qquad\qquad + \frac{1}{2t} \big(\mathrm{R}_{K_z}^2(t)-1\big)\,\W_2^2(\rho_0,\rho_1) + \Ent_\meas(\rho_0) - \Ent_\meas(\WHeat_t\rho_1) \\
&\qquad\qquad\qquad\qquad - \frac{1}{t}\int_0^1 s \int_0^t\int_\mms \Gamma(Q_s\phi_t)\, \Schr_{s(t-r)}^{2(\k- K_z)}\!\left((\k-K_z)\, \ChHeat_{sr}f_{\vartheta_t(s)}\right)\d\meas\d r \d s,
\end{align*}
where we have put  $\vartheta_t:=\vartheta_{K_z,t}$.
Note that the $\limsup$ as $t\downarrow 0$ of the last term is nonnegative since 
$(\k-K_z)\,f_s \geq 0$ $\meas$-a.e.~on $\mms$ for every $s\in[0,1]$ and 
\begin{equation*}
\lim_{t\downarrow 0} \frac{1}{t}\int_0^t \Schr_{s(t-r)}^{2(\k-K_z)}\!\left((\k-K_z)\, \ChHeat_{sr}f_{\vartheta_t(s)}\right)\d r = (k-K_z)\,f_s
\end{equation*}
w.r.t.~convergence in $\Ell^1(\mms,\meas)$. Indeed,  $\vartheta_t(s) \to s$ as $t\downarrow 0$ for every $s\in [0,1]$  and therefore $f_{\vartheta_t(s)} \to f_s$ pointwise $\meas$-a.e. As all considered functions are nonnegative and $\int_\mms f_{\vartheta_t(s)}\d\meas = \int_\mms f_s\d\meas$ for all $t>0$, we have $f_{\vartheta_t(s)} \to f_s$ in $\Ell^1(\mms,\meas)$ as $t\downarrow 0$. We conclude by strong continuity of the heat and the Schrödinger semigroup with potential $2(\k-K_z)$ in $\Ell^1(\mms,\meas)$.

Lower semicontinuity of $\Ent_\meas$ yields $-\liminf_{t\downarrow 0} \Ent_\meas(\WHeat_t\rho_1) \leq -\Ent_\meas(\rho_1)$,
and clearly $\mathrm{R}_{K_z}^2(t) = 1-K_zt + \mathrm{o}(t)$ as $t\downarrow 0$. Lastly, observe that
$\smash{\big(\W_2^2(\WHeat_t\rho_1,\rho_0) - \W_2^2(\WHeat_t^*\rho_1,\rho_0)\big)}/2t \to 0$
according to  Lemma \ref{lma-trunc} applied with $a:= 2$.
Thus, we finally deduce
\begin{align*}
\limsup_{t\downarrow0}\frac{1}{2t}\left(\W_2^2(\WHeat_t\rho_1,\rho_0) - \W_2^2(\rho_0,\rho_1)\right) + \frac{K_z}{2} \W_2^2(\rho_0,\rho_1) \leq \Ent_\meas(\rho_0) - \Ent_\meas(\rho_1).\tag*{\qedhere}
\end{align*}
\end{proof}

\begin{thm}\label{Thm:GE2 to CD} The $2$-gradient estimate $\GE_2(\k)$ implies $\CD(\k,\infty)$.
\end{thm}

\begin{proof} Given $\varepsilon > 0$, Proposition \ref{local evi} translates into a ``local'' $\EVI(\k-\varepsilon)$ property at time $0$: for every $z\in\mms$, choosing $\delta > 0$ and $K_z\in\R$ such that $K_z \leq \k\leq K_z+\varepsilon$ in $B_{2\delta}(z)$, we obtain that for all $\mu,\nu\in\Prob_2(\mms)\cap\Dom(\Ent_\meas)$ with support in $B_\delta(z)$ and bounded densities w.r.t.~$\meas$, for $\boldsymbol{\pi}\in\Prob(\Geo(\mms))$ representing the $\W_2$-geodesic from $\mu$ to $\nu$, we have
\begin{equation*}
\frac{\diff^+}{\diff t}\bigg\vert_{t=0}\frac{1}{2}\W_2^2(\WHeat_t\mu,\nu) + \int_0^1\int_{\Geo(\mms)} (1-s)\,\big(\k(\gamma_s) - \varepsilon\big)\,\vert\dot\gamma\vert^2\d\boldsymbol{\pi}(\gamma)\d s \leq \Ent_\meas(\nu) - \Ent_\meas(\mu).
\end{equation*}
With the same argument used in the proof of \cite[Theorem 3.4]{sturm2015} for the equivalence of $\CD(\k,\infty)$ and $\EVI(\k)$, we conclude that this local $\EVI(\k-\varepsilon)$ implies a ``local'' $\CD(\k-\varepsilon,\infty)$ condition in the following sense: for all $z\in\mms$ there exists $\delta > 0$ such that for all $\mu_0,\mu_1\in\Prob_2(\mms)\cap\Dom(\Ent_\meas)$ with support in $B_\delta(z)$ and bounded densities w.r.t.~$\meas$, if $\boldsymbol{\pi}\in\Prob(\Geo(\mms))$ represents the $\W_2$-geodesic from $\mu_0$ to $\mu_1$, for every $t\in [0,1]$, we have
\begin{equation*}
\Ent_\meas(\mu_t) \leq (1-t) \Ent_\meas(\mu_0) + t\Ent_\meas(\mu_1) - \int_0^1\int_{\Geo(\mms)} \green(s,t)\,\big(\k(\gamma_s)-\varepsilon\big) \,\vert\dot\gamma\vert^2\d\boldsymbol{\pi}(\gamma)\d s.
\end{equation*}

Using the local-to-global property from \cite[Theorem 3.7]{sturm2015} and taking the limit $\varepsilon\downarrow 0$, noticing again that the choice of $\W_2$-geodesics does not depend on $\varepsilon$, allows us to pass from this local $\CD(\k-\varepsilon,\infty)$ property to $\CD(\k-\varepsilon,\infty)$ and finally to $\CD(\k,\infty)$.
\end{proof}

\subsection[From $\EVI$ to a differential $2$-transport estimate]{From {\boldmath{$\EVI$}} to a differential {\boldmath{2}}-transport estimate}\label{Sec:3.3}

It has already been observed in \cite{ketterer2015} that  $\EVI(\k)$ yields contraction estimates for the $2$-Wasserstein distance along two heat flows starting at regular measures. 
For irregular initial data, we now aim in deducing a weak version of it, see also Remark \ref{Re:WDTE}.

\begin{Prop}\label{Pr:EVI implies DTE} The $\EVI(\k)$ implies the following differential $2$-transport estimates:
\begin{enumerate}[label=\textnormal{(\roman*)}]
\item for every $\mu_1,\mu_2\in\Prob_2(\mms) \cap \Dom(\Ent_\meas)$, one has
\begin{equation}\label{the estimate}
\frac{\diff^+}{\diff t}\bigg\vert_{t=0}\W_2^2(\WHeat_t\mu_1,\WHeat_t\mu_2)\leq -2\int_0^1\int_{\Geo(\mms)}\k(\gamma_s)\,\vert\dot\gamma\vert^2\d\boldsymbol{\pi}(\gamma)\d s,
\end{equation}
where $\boldsymbol{\pi}\in\Prob(\Geo(\mms))$ represents the $\W_2$-geodesic from $\mu_1$ to $\mu_2$, and
\item for all $x,y\in\mms$,
\begin{equation*}
\frac{\diff^+}{\diff t}\bigg\vert_{t=0} \W_2^2(\WHeat_t\delta_x,\WHeat_t\delta_y) \leq -2\kk(x,y)\met^2(x,y).
\end{equation*}
\end{enumerate}
\end{Prop}

\begin{proof} Concerning (i), up to truncating $\k$ and using monotone convergence afterwards, we may assume that $\k$ is bounded. Then the claim follows by adding up the $\EVI(\k)$, integrated from $t$ to $t+h$, $h> 0$, for the flow $(\WHeat_t\mu_1)_{t\ge0}$ with observation point $\WHeat_{t+h}\mu_2$ and for the flow $(\WHeat_t\mu_2)_{t\ge0}$ with observation point $\WHeat_t\mu_1$. The entropy terms cancel out, and we obtain the desired estimate by dividing by $h$ and letting $h\downarrow 0$.
Some care, however, is requested to deal with the double $t$-dependence of the nonsmooth function $t\mapsto \smash{W_2^2\big(\WHeat_t\rho_{0},\WHeat_t\rho_{1}\big)}$. This has been addressed in \cite[Theorem 6.1]{ketterer2015}. 

Next, we show (ii). Denote by $\kk_n\in\Lip_\bdd(\mms\times\mms)$ a sequence converging pointwise from below in a monotone way to $\kk$, see Lemma \ref{Le:kk approximation}, and put $\k_n(x) := \kk_n(x,x)$ for $x\in\mms$. Given $x,y\in\mms$ and $t>0$, select $\tau_*>0$ small enough so that, for every $\tau\in (0, \tau_*)$,
\begin{equation*}
\W_2^2(\WHeat_\tau\delta_x,\WHeat_\tau\delta_y) \leq \met^2(x,y) + 2t^2.
\end{equation*}
The local absolute continuity of the curves $(\WHeat_t\delta_x)_{t\geq 0}$ and $(\WHeat_t\delta_y)_{t\geq 0}$ on $(0,\infty)$ w.r.t.~$\W_2$ and property (i) with $\k_n$ in place of $\k$, since $\k_n \leq \k$ on $\mms$, yield
\begin{align*}
\frac{1}{2t}\left(\W_2^2(\WHeat_t\delta_x, \WHeat_t\delta_y) - \met^2(x,y)\right) &\leq t + \frac{1}{2t}\left(\W_2^2(\WHeat_t\delta_x,\WHeat_t\delta_y) - \W_2^2(\WHeat_\tau\delta_x,\WHeat_\tau\delta_y)\right)\\
&\leq t - \frac{1}{t}\int_\tau^t\int_0^1\int_{\Geo(\mms)} \k_n(\gamma_s)\,\vert\dot{\gamma}\vert^2\d\boldsymbol{\pi}_r(\gamma)\d s\d r,
\end{align*}
where $\boldsymbol{\pi}_r\in\Prob(\Geo(\mms))$ represents the $\W_2$-geodesic from $\WHeat_r\delta_x$ to $\WHeat_r\delta_y$. As $n\to\infty$, by monotone convergence, the above inequality still holds with $\k$ in place of $\k_n$. Thus, the definition of $\kk$ and the inequality $\kk_n \leq \kk$ on $\mms$ for every $n\in\N$ give, setting $\pi_r := (\eval_0,\eval_1)_\push\boldsymbol{\pi}_r$, 
\begin{equation*}
\frac{1}{2t}\left(\W_2^2(\WHeat_t\delta_x, \WHeat_t\delta_y) - \met^2(x,y)\right) \leq t - \frac{1}{t}\int_\tau^t\int_{\mms\times\mms} \kk_n(x',y')\met^2(x',y')\d\pi_r(x',y')\d r.
\end{equation*}

Since $\WHeat_r\delta_x \to \delta_x$ and $\WHeat_r\delta_y \to \delta_y$ w.r.t.~$\W_2$ as $r\to 0$ and since $\W_2(\WHeat_r\delta_x,\WHeat_r\delta_y)$ is bounded uniformly in for small $r$, stability of optimal couplings, see \cite[Proposition 7.1.3]{ambrosio2008}, and uniqueness of the $\W_2$-optimal coupling $\pi_0 := \delta_x\otimes\delta_y$ imply that $\pi_r \to \pi_0$ weakly as $r\to 0$. Thus, the map
$r \mapsto \int_{\mms\times\mms}\kk_n\met^2\d\pi_r$
is continuous at $0$ by \cite[Lemma 4.3]{villani2009}. 
The claim follows by taking successively $\tau\downarrow 0$, $t\downarrow 0$ and $n\to\infty$ in the above inequality.
\end{proof}

A posteriori, knowing from Theorem \ref{Th:Huge thm} that $\EVI(\k)$ implies $\GE_1(\k)$, we will be able to improve the bound (ii) from Proposition \ref{Pr:EVI implies DTE} even for exponents different from $2$, see Remark \ref{sup improvement} below.

\section[Duality of \boldmath{$p$}-transport estimates and \boldmath{$q$}-gradient estimates]{Duality of \boldmath{$p$}-transport estimates and \boldmath{$q$}-gradient estimates}
\label{sec:transport}
Throughout the rest of this article, given $t\geq 0$, we use the short-hand notation $\PS{t} := \Cont([0,t];\mms\times\mms)$. Moreover, at several instances we consider a function $\kkkk\colon\mms\times\mms\to\R$ which, unless stated otherwise, is assumed lower semicontinuous and lower bounded. However, it should practically rather be thought of as a bounded Lipschitz function ``approximating'' $\kk$ from below without being of the particular form \eqref{k hut}. This often allows us to assume that $\kkkk \in \Lip_\bdd(\mms\times\mms)$, while $\kk$ is not continuous in general, even if $\k$ is Lipschitz.

\subsection{Perturbed costs and coupled Brownian motions}\label{Sec:4.1}

Given any $p\in [1,\infty)$ and  $\mu_1,\mu_2\in\Prob_p(X)$, let us define the \emph{perturbed $p$-transport cost with potential $-p\kkkk$} at $t\geq 0$ by 
\begin{equation}\label{pert cost}
W_p^{\kkkk}(\mu_1,\mu_2,t):=\inf_{(\PP,\B^1,\B^2)} \EE\Big[\expo^{\int_0^{2t} p \kkkk\left(\B^1_{r}, \B^2_{r}\right)/2\d r} \met^p\!\big(\B^1_{2t},\B^2_{2t}\big)\Big]^{1/p},
\end{equation} 
where the infimum is taken over all pairs of coupled Brownian motions $\smash{\big(\PP,\B^1\big)}$ and $\big(\PP,\B^2\big)$ on $\mms$, restricted to $[0,2t]$ and modeled on a common probability space, with initial distributions $\mu_1$ and $\mu_2$, respectively. In more analytic words, 
\begin{equation}\label{pert cost2}
W_p^{\kkkk}(\mu_1,\mu_2,t)=\inf_{\boldsymbol{\nu}}  \Big(\int_{\PS{2t}} \expo^{\int_0^{2t} p\kkkk\left(\gamma_{r}^1,\gamma_{r}^2\right)/2\d r} \met^p\!\big(\gamma_{2t}^1,\gamma_{2t}^2\big)\d\boldsymbol{\nu}(\gamma)\Big)^{1/p}\!,
 \end{equation} 
the infimum being taken over all $\boldsymbol{\nu}\in\Prob(\PS{2t})$ whose marginals $\boldsymbol{\nu}_1,\boldsymbol{\nu}_2 \in \Prob(\Cont([0,2t];\mms))$ are the laws of Brownian motions on $\mms$, restricted to $[0,2t]$, with initial distribution $\mu_1$ and $\mu_2$, respectively. If $\kkkk = \kk$, this is the usual perturbed $p$-transport cost from Section \ref{Transport def}.

A natural, albeit nontrivial identity relates the perturbed $p$-transport cost  in the case of constant  $\k$ with the usual $p$-transport cost.

\begin{lma}\label{Le:Constant perturbed cost} If $\kkkk$ is constantly equal to $L\in\R$ then, for $t\geq 0$,
\begin{equation*}
\W_p^{\kkkk}(\mu_1,\mu_2,t)=\expo^{Lt}\, \W_p(\WHeat_t\mu_1,\WHeat_t\mu_2).
\end{equation*}
\end{lma}

\begin{proof} Since $\smash{\W_p(\WHeat_t\mu_1,\WHeat_t\mu_2)^{1/p} = \inf_{(\mathsf{x},\mathsf{y})}\EE\big[\!\met^p(\mathsf{x},\mathsf{y})\big]^{1/p}}$, the infimum ranging  over pairs of random variables $\mathsf{x}\sim\WHeat_t\mu_1$ and $\mathsf{y}\sim\WHeat_t\mu_2$ defined on a common probability space $(\Omega,\mathscr{F},\PP)$, and as $\smash{\B_{2t} \sim \WHeat_t\mu}$ for every Brownian motion $(\PP,\B)$ with initial distribution $\mu\in\Prob(\mms)$, we get 
\begin{equation*}
\W_p^{\kkkk}(\mu_1,\mu_2,t)\geq \expo^{Lt}\, \W_p(\WHeat_t\mu_1,\WHeat_t\mu_2).
\end{equation*}

For the converse inequality, let $\pi_t\in\Prob(\mms\times\mms)$ be a $W_p$-optimal coupling of $\WHeat_t\mu_1$ and $\WHeat_t\mu_2$. Consider Brownian motions $\smash{\big(\PP_1,\B^1\big)}$ and $\smash{\big(\PP_2,\B^2\big)}$, restricted to $[0,2t]$, starting at $\mu_1$ and $\mu_2$, defined on probability spaces $(\Omega_1,\mathscr{F}_1,\PP_1)$ and $(\Omega_2,\mathscr{F}_2,\PP_2)$, respectively. Define the ``bridge measures'' $\PP_1^x$ for $x\in\mms$ by disintegrating $\PP_1$ w.r.t.~$\WHeat_t\mu_1(\d x)$ or, in other words, by conditioning $\B^1$ on the event $\{\B_{2t}^1=x\}$. Similarly, let $\PP_2^y$ for $y\in\mms$ be the disintegration of $\PP_2$ w.r.t.~$\WHeat_t\mu_2(\d y)$.
Consider the ``glued measure'' $\smash{\widetilde{\PP}}$ defined by
\begin{equation*}
\widetilde{\PP}:=\int_{\mms\times\mms}\PP_1^x\otimes \PP_2^y\d\pi_t(x,y)
\end{equation*}
on $\Omega:=\Omega_1\times \Omega_2$. Then $\smash{\big(\widetilde{\PP},\B^1\big)}$ and $\smash{\big(\widetilde{\PP},\B^2\big)}$ is a 
pair of coupled Brownian motions with joint distribution $\pi_t$ at time $2t$. The desired inequality then follows directly, since
\begin{equation*}
\widetilde{\EE}\big[\!\met^p\!\big(\B_{2t}^1,\B_{2t}^2\big)\big] = \int_{\mms\times\mms}\met^p(x,y)\d\pi_t(x,y) = \W_p^p(\WHeat_t\mu_1,\WHeat_t\mu_2).\qedhere
\end{equation*}
\end{proof}

\begin{lma}\label{min att} For every $p\in[1,\infty)$, $t\geq0$ and $\mu_1,\mu_2\in\Prob_p(\mms)$ as above, the infima in \eqref{pert cost} and in \eqref{pert cost2} are attained. 

Moreover, for every sequence of lower semicontinuous functions $\kkkk_n \colon \mms\times\mms\to\R$ converging pointwise to $\kkkk$ from below in an increasing way, we have
\begin{equation*}
\lim_{n\to\infty} \W_p^{\kkkk_n}(\mu_1,\mu_2,t) = \W_p^\kkkk(\mu_1,\mu_2,t).
\end{equation*}
\end{lma}

\begin{proof} The lower semicontinuity of $\kkkk$ implies the one of 
 \begin{equation*}
\gamma\quad \longmapsto\quad \expo^{\int_0^{2t} p\kkkk\left(\gamma_{r}^1,\gamma_{r}^2\right)/2\d r}\met^p\!\big(\gamma_{2t}^1,\gamma_{2t}^2\big)
\end{equation*}
w.r.t.~the uniform topology on $\PS{2t}$ which in turn implies weak lower semicontinuity of
\begin{equation*}
\boldsymbol{\nu}\quad \longmapsto\quad \
 \int_{\PS{2t}} \expo^{\int_0^{2t} p\kkkk\left(\gamma_{r}^1,\gamma_{r}^2\right)/2\d r} \met^p\!\big(\gamma_{2t}^1,\gamma_{2t}^2\big)\d\boldsymbol{\nu}(\gamma)
 \end{equation*} 
in $\Prob(\PS{2t})$. This gives the existence of a minimizer for \eqref{pert cost2} by a standard argument
since, according to \cite[Lemma 4.4]{villani2009}, the family of $\boldsymbol{\nu}\in\Prob(\PS{2t})$ with given marginals is tight as the sets of marginals are both singletons.

The second assertion is a standard argument via $\Gamma$-convergence of the functionals whose infima give $\smash{\W_p^{\kkkk_n}(\mu_1,\mu_2,t)}$ and $\smash{\W_p^\kkkk(\mu_1,\mu_2,t)}$, respectively, in $\Prob(\PS{2t})$.
\end{proof}

Let us denote by $\Borel^\nu(\mms\times\mms)$ the completion of the Borel $\sigma$-field on $\mms\times\mms$ w.r.t.~a given  $\nu\in\Prob(\mms\times\mms)$, and then
\begin{equation*}
\Borel^\mathrm{univ}(\mms\times\mms) := \bigcap_{\nu\in\Prob(\mms\times\mms)} \Borel^\nu(\mms\times\mms)
\end{equation*}
is the $\sigma$-field of all \emph{universally measurable} subsets of $\mms\times\mms$.

\begin{lma}\label{kernel}
For every $t\geq 0$ and $p\in [1,\infty)$, there exists a universally measurable map
\begin{equation*}
\boldsymbol{\eta}^t\colon\quad \mms\times\mms \quad\longrightarrow\quad \Prob(\PS{2t})
\end{equation*}
such that for every $x,y\in X$, the marginals of $\boldsymbol{\eta}_{x,y}^t := \boldsymbol{\eta}^t(x,y)$ are laws of Brownian motions, restricted to $[0,2t]$, starting in $x$ and $y$, respectively, and $\smash{\boldsymbol{\eta}_{x,y}^t}$ is a minimizer in the definition \eqref{pert cost2} of $\smash{\W_p^\kkkk(\delta_x,\delta_y,t)}$.
\end{lma}

\begin{proof}  According to Lemma \ref{min att}, for each pair $(x,y)\in\mms\times\mms$ there exists an admissible measure on $\Prob(\PS{2t})$ which attains the infimum in \eqref{pert cost2}. The class of all probability measures with this property 
is closed. 
Then a measurable selection argument, see \cite{bogachev2007, sturm2015}, allows us to produce a family of measures $\boldsymbol{\eta}^t_{x,y}$ still satisfying the minimality property so that $(x,y) \mapsto \boldsymbol{\eta}^t_{x,y}$ is universally measurable in $(x,y)\in\mms\times\mms$. 
\end{proof}

An important consequence of these observations is a type of Markov property which will be crucial in the proof of Theorem \ref{thm:nonincreasing}. For this and also for later use, fix $s,t\geq 0$, a measure $\boldsymbol{\nu}\in\Prob(\PS{s})$ and a universally measurable map $\boldsymbol{\mu} \colon \mms\times\mms \to \Prob(\PS{t})$ such that $(\eval_0)_\push\boldsymbol{\mu}_{x,y} = \delta_x\otimes \delta_y$ for all $x,y\in\mms$. Define their composition $\boldsymbol{\mu} \circ\boldsymbol{\nu} \in \Prob(\PS{s+t})$ by
\begin{equation*}
\int_{\PS{s+t}} f(\gamma)\d (\boldsymbol{\mu} \circ\boldsymbol{\nu})(\gamma) :=
\int_{\PS{s}}\int_{\PS{t}} f\big(\Phi_{s,t}(\alpha,\beta)\big)\d\boldsymbol{\mu}_{\alpha_s^1,\alpha_s^2}(\beta)\d\boldsymbol{\nu}(\alpha) \quad \text{for every } f\in \Cont_\bdd(\PS{s+t}),
\end{equation*}
where
\begin{equation*}
\Phi_{s,t}(\alpha,\beta)_r := \alpha_r \quad\text{if } r\in [0,s]\quad\text{and}\quad \Phi_{s,t}(\alpha,\beta)_r := \beta_{r-s}\quad\text{if } r\in (s,s+t]
\end{equation*}
denotes the \emph{concatenation map} ``gluing'' together the curves $(\alpha_\sigma)_{\sigma\in [0,s]}$ and $(\beta_\tau)_{\tau\in[0,t]}$.

\begin{Prop}\label{markov} For every $p\in [1,\infty)$, every $s,t\geq 0$ and all $\mu_1,\mu_2\in\Prob_p(\mms)$, there exists a 
pair $\smash{\big(\PP,\B^1\big)}$ and $\smash{\big(\PP,\B^2\big)}$ of coupled Brownian motions on $\mms$
 with initial distributions $\mu_1$ and $\mu_2$, respectively, which minimizes \eqref{pert cost} for the given time $t$ and such that
\begin{equation}\label{Markov prop inequality}
\W_p^{\kkkk}(\mu_1,\mu_2,t+s)^p\le \EE\Big[\expo^{\int_0^{2t} p \kkkk\left(\B^1_{r}, \B^2_{r}\right)/2\d r}\, \W_p^{\kkkk}\big(\delta_{\B^1_{2t}},\delta_{\B^2_{2t}},s\big)^p \Big].
\end{equation}
\end{Prop}

\begin{proof} Denote the map from Lemma \ref{kernel} with $s$ in place of $t$ by $\boldsymbol{\eta}^s$, denote a minimizer of \eqref{pert cost2} for time $t$ by $\boldsymbol{\nu}_t$, and define $\boldsymbol{\eta}^{t+s} := \boldsymbol{\eta}^s\circ\boldsymbol{\nu}_t\in\Prob(\PS{2(s+t)})$. This defines a coupling of the laws of two Brownian motions with initial distributions $\mu_1$ and $\mu_2$, respectively, restricted to $[0,2(t+s)]$ such that 
\begin{align*}
&\int_{\PS{2(s+t)}} \expo^{\int_0^{2(t+s)} p\kkkk\left(\gamma_{r}^1,\gamma_{r}^2\right)/2\d r}\met^p\!\big(\gamma_{2(t+s)}^1,\gamma_{2(t+s)}^2\big)\d\boldsymbol{\nu}^{t+s}(\gamma)\\
&\qquad\qquad =
\int_{\PS{2t}} \expo^{\int_0^{2t} p\kkkk\left(\alpha_{r}^1,\alpha_{r}^2\right)/2\d r}\,
W_p^\k\big(\delta_{\alpha^1_{2t}}, \delta_{\alpha^2_{2t}}, s\big)^p\d\boldsymbol{\nu}_{t}
(\alpha).
\end{align*}
This proves the claim.
\end{proof}

Less formally, the previous construction can be described as follows. To estimate the perturbed $p$-transport cost at time $t+s$, we construct the required process by first choosing a 
 pair process $\smash{\big(\B^1,\B^2\big)}$ of Brownian motions with given initial distributions $\mu_1$ and $\mu_2$ which realizes the minimum for $\smash{\W_p^\kkkk(\mu_1,\mu_2,t)}$. Then we switch to a pair of Brownian motions starting in $\B^1_{2t}$ and $\B^2_{2t}$, respectively, which minimizes the cost at time $s$.

\subsection[From weak differential $p$-transport inequalities to $p$-transport estimates]{From differential \boldmath{$p$}-transport inequalities to \boldmath{$p$}-transport estimates}\label{Sec:4.2}

To deduce a $p$-transport estimate $\PTE_p(\k)$, we have to control the upper derivatives of the function $\smash{t\mapsto \W_p^\kk(\delta_x,\delta_y,t)^p}$ or, more generally, of $\smash{t\mapsto \W_p^\kkkk(\delta_x,\delta_y,t)^p}$ for $x,y\in\mms$. 

\begin{lma}\label{diff w} Assume that $\kkkk\in\Cont_\bdd(\mms\times\mms)$. Then for all $x,y\in\mms$ and $p\in [1,\infty)$, we have
$$  \frac{\diff^+}{\diff t}\bigg\vert_{t=0} W_p^\kkkk(\delta_x,\delta_y,t)^p\le p\,\kkkk(x,y) \met^p(x,y)+
\frac{\diff^+}{\diff t}\bigg\vert_{t=0}W_p^p(\WHeat_t\delta_x,\WHeat_t\delta_y)
.$$
\end{lma}

\begin{proof} Choose any exponent $p'\in (p,\infty)$ with dual exponent $q'\in (1,\infty)$. For all $t>0$, denote by $\smash{\big(\PP,\B^1\big)}$ and $\smash{\big(\PP,\B^2\big)}$ a pair of coupled Brownian motions starting in $(x,y)$ and such that the law of $\smash{\big(\B_{2t}^1,\B_{2t}^2\big)}$ constitutes a $\W_{p'}$-optimal coupling of $\WHeat_t\delta_x$ and $\WHeat_t\delta_y$. Albeit this process still depends on $t$,  we  suppress this dependence in the sequel to simplify the notation.
For a precise construction of such process, we refer to the proof of Lemma \ref{Le:Constant perturbed cost}.

Observe that
\begin{align*}
&\frac{\diff^+}{\diff t}\bigg\vert_{t=0}W_p^\kkkk(\delta_x,\delta_y,t)^p\\
&\qquad\qquad\leq \limsup_{t\downarrow 0} \frac{1}{t}\EE\Big[\expo^{\int_0^{2t}p\kkkk\left(\B_{r}^{1},\B_{r}^{2}\right)/2\d r} \met^p\!\big(\B_{2t}^{1},\B_{2t}^{2}\big) - \met^p\!\big(\B_{0}^{1},\B_{0}^{2}\big)\Big]\\
&\qquad\qquad \leq \limsup_{t\downarrow 0} \frac{1}{t}\EE\bigg[\Big(\expo^{\int_0^{2t}p\kkkk\left(\B_{r}^{1},\B_{r}^{2}\right)/2\d r} - 1\Big) \met^p\!\big(\B_{2t}^{1},\B_{2t}^{2}\big)\bigg] + \frac{\diff^+}{\diff t}\bigg\vert_{t=0} \W_{p'}^p(\WHeat_t\delta_x,\WHeat_t\delta_y).
\end{align*}
Each of the last two limits will be estimated separately. The last term will converge to the upper derivative of $\smash{\W_p^p(\WHeat_t\delta_x,\WHeat_t\delta_y)}$ at $0$ as $p'\downarrow p$ by monotone convergence.
Moreover, since $\kkkk$ is bounded, the former term can be estimated through
\begin{align*}
&\limsup_{t\downarrow 0} \frac{1}{t}\EE\bigg[\Big(\expo^{\int_0^{2t}p\kkkk\left(\B_{r}^{1},\B_{r}^{2}\right)/2\d r} - 1\Big) \met^p\!\big(\B_{2t}^{1},\B_{2t}^{2}\big)\bigg] \leq \limsup_{t\downarrow 0} \frac{p}{2t}\EE\Big[\int_0^{2t}\kkkk\big(\B_{r}^{1},\B_{r}^{2}\big)\d r 
 \met^p\!\big(\B_{2t}^{1},\B_{2t}^{2}\big)\Big].
\end{align*}
Now we split the expectation into a term where $\big(\B^1,\B^2\big)$ behaves well and a remainder term. Let $\varepsilon > 0$ and choose $\delta > 0$ such that
\begin{equation*}
\max\!\Big\lbrace\big\vert\kkkk(x',y') - \kkkk(x,y)\big\vert, \big\vert \met^p(x',y')- \met^p(x,y)\big\vert\Big\rbrace\leq \varepsilon\quad\text{for every }x'\in B_\delta(x),\ \! y'\in B_\delta(y),
\end{equation*}
and define the exceptional set $E_{r,2t}$ for $r\in (0,2t)$ by
\begin{equation*}
E_{r,2t} := \big\lbrace \B_{r}^{1}\notin B_\delta(x)\big\rbrace\cup \big\lbrace \B_{2t}^{1} \notin B_\delta(x)\big\rbrace \cup\big\lbrace \B_{r}^{2}\notin B_\delta(y)\big\rbrace\cup \big\lbrace \B_{2t}^{2} \notin B_\delta(y)\big\rbrace.
\end{equation*}
By these definitions and Fubini's theorem, since $\kkkk$ is bounded,
\begin{align*}
&\limsup_{t\downarrow 0} \frac{p}{2t}\EE\Big[\int_0^{2t}\kkkk\big(\B_{r}^{1},\B_{r}^{2}\big)
 \, \One_{E_{r,2t}^\mathrm{c}}\d r\met^p\!\big(\B_{2t}^{1},\B_{2t}^{2}\big)\Big]\\
&\qquad\qquad \leq p\,\big(\kkkk(x,y) + \varepsilon\big)\, \big(\!\met^p(x,y) + \varepsilon\big) \, \limsup_{t\downarrow 0}\frac{1}{2t}\int_0^{2t} \PP\big[E_{r,2t}^\mathrm{c}\big]\d r.
\end{align*}
According to Lemma \ref{tail estimate}, we have $\PP[E_{r,2t}] \to 0$ as $r\downarrow0$ and $t \downarrow 0$, therefore the latter $\limsup$ is equal to $1$. On the other hand, if $C > 0$ denotes an upper bound for $\kkkk$, using Hölder's inequality the second term can be bounded through
\begin{align*}
&\limsup_{t\downarrow 0} \bigg\vert\frac{p}{2t}\EE\Big[\int_0^{2t} \kkkk\big(\B_{r}^{1},\B_{r}^{2}\big)\, \One_{E_{r,2t}}\d r\,\met^p\!\big(\B_{2t}^{1},\B_{2t}^{2}\big)\Big]\bigg\vert\\
&\qquad\qquad \leq p\,C \, \limsup_{t\downarrow 0} \EE\Big[\!\met^{p'}\!\!\big(\B_{2t}^{1},\B_{2t}^{2}\big)\Big]^{p/p'} \, \limsup_{t\downarrow 0}\Big( \frac{1}{2t}\int_0^{2t} \PP\big[E_{r,2t}\big]\d r\Big)^{1-p/p'}.
\end{align*}
By the choice of the pair process $\big(\B^1,\B^2\big)$, the first $\limsup$ is equal to $\met^p(x,y)$ while the second one is $0$, as already observed above. Since $\varepsilon$ was arbitrary, we obtain the claim.
\end{proof}

\begin{thm}\label{thm:nonincreasing}
Fix $p\in [1,\infty)$ and assume the differential $p$-transport estimate
\begin{equation}\label{Eq:WDTE assumption}
\frac{\diff^+}{\diff t}\bigg\vert_{t=0} \W_p^p(\WHeat_t\delta_x,\WHeat_t\delta_y) \leq -p\,\kk(x,y)\met^p(x,y)\quad\text{for every }x,y\in\mms.
\end{equation}
Then the $p$-transport estimate $\PTE_p(k)$ is satisfied.
\end{thm}

\begin{proof} We first show that for all $\mu_1,\mu_2\in\Prob_p(\mms)$, the function $\smash{t\mapsto \W_p^\kkkk(\mu_1,\mu_2,t)}$ is nonincreasing on $[0,\infty)$ whenever $\kkkk\in\Cont_\bdd(\mms\times\mms)$ with $\kkkk\leq\kk$ on $\mms\times\mms$.

To get started, we demonstrate that its $p$-th power $\smash{t\mapsto W_p^\kkkk(\mu_1,\mu_2,t)^p}$ is upper Lipschitz continuous on $[0,\infty)$. To see this, fix $h \in (0,1]$ and $t> 0$, and consider the  pair process $\smash{\big(\B^1,\B^2)}$ as provided by Proposition \ref{markov}. 
By the estimate \eqref{Markov prop inequality} of this proposition, Lemma \ref{Le:Constant perturbed cost} and contractivity of the Wasserstein heat flow, we have
\begin{align}
&\frac{1}{h}\left(\W_p^\kkkk(\mu_1,\mu_2, t+h)^p - \W_p^\kkkk(\mu_1,\mu_2,t)^p\right)\nonumber\\
&\qquad\qquad \leq \frac{1}{h}\EE\bigg[\expo^{\int_0^{2t}p\kkkk\left(\B_{r}^1,\B_{r}^2\right)/2\d r}\Big(\W_p^\kkkk\big(\delta_{\B_{2t}^1},\delta_{\B_{2t}^1},h\big)^p - \met^p\!\big(\B_{2t}^1,\B_{2t}^2\big)\Big)\bigg]\label{take h to 0}\\
&\qquad\qquad \leq \frac{1}{h}\EE\bigg[\expo^{\int_0^{2t}p\kkkk\left(\B_{r}^1,\B_{r}^2\right)/2\d r}\met^p\!\big(\B_{2t}^1,\B_{2t}^2\big)\,\big(\expo^{pCh} - 1\big)\bigg] \le C'\,\W_p^\kkkk(\mu_1,\mu_2,t)^p\nonumber
\end{align}
for suitable nonnegative constants $C$ and $C'$. This proves upper Lipschitz continuity of the $p$-th power of the perturbed $p$-transport cost with potential $-p\kkkk$, which in turn implies
\begin{equation}\label{int w}
W_p^\kkkk(\mu_1,\mu_2,\tau)^p-W_p^\kkkk(\mu_1,\mu_2,\sigma)^p\le \int_\sigma^\tau
\frac{\diff^+}{\diff t}W_p^\kkkk\big(\mu_1,\mu_2,t\big)^p \d t
\end{equation}
for every $\sigma,\tau \in [0,\infty)$ with $\sigma \leq \tau$. Letting $h\downarrow 0$, the estimate \eqref{take h to 0} and the observation
\begin{equation*}
\EE\Big[\expo^{\int_0^{2t}p\kkkk\left(\B_{r}^1,\B_{r}^2\right)/2\d r} \met^p\!\big(\B_{2t}^1,\B_{2t}^2\big)\Big] < \infty,
\end{equation*}
which justifies to apply Fatou's lemma, give
\begin{equation*}
\frac{\diff^+}{\diff t}\W_p^\kkkk(\mu_1,\mu_2,t)^p \leq 
 \EE\bigg[\expo^{\int_0^{2t} p\kkkk\left(\B_{r}^1,\B_{r}^2\right)/2\d r}\, \frac{\diff^+}{\diff h}\bigg\vert_{h=0}\W_p^\kkkk\big(\delta_{\B_{2t}^1},\delta_{\B_{2t}^2}, h\big)^{p}\bigg].
\end{equation*}
Finally, the inequality \eqref{int w} for the upper derivative inside the expectation, Lemma \ref{diff w} and then the assumed estimate \eqref{Eq:WDTE assumption}, noting that $-\kk \leq -\kkkk$ on $\mms\times\mms$, yield the initial claim.

The nonincreasingness of $\smash{t\mapsto \W_p^\kk(\mu_1,\mu_2,t)}$ on $[0,\infty)$ is then immediate due to an easy approximation argument using Lemma \ref{Le:kk approximation} and Lemma \ref{min att}.
\end{proof}

\begin{cor}\label{Cor:TE ==> WDTE} For every $p\in[1,\infty)$, $\PTE_p(\k)$ implies
\begin{equation*}
\frac{\diff^+}{\diff t}\bigg\vert_{t=0} \W_p^p(\WHeat_t\delta_x,\WHeat_t\delta_y)\leq -p\,\kk(x,y)\met^p(x,y)\quad\text{for every }x,y\in\mms.
\end{equation*}
In particular, $\PTE_p(\k)$ and the differential $p$-transport estimate \eqref{Eq:WDTE assumption} are equivalent.
\end{cor}

\begin{proof} Fix $x,y\in\mms$. For every $t > 0$ and $p' \in (p,\infty)$, we denote by $\smash{\big(\PP,\B^1\big)}$ and $\smash{\big(\PP,\B^2\big)}$ a pair of coupled Brownian motions which realizes the minimum in the definition of $\smash{\W_{p'}^\kkkk(\delta_x,\delta_y,t)}$. This process does depend on $t$, but we leave out this dependency from the notation. Arguing as in the proof of Lemma \ref{diff w}, we get
\begin{align*}
&\frac{\diff^+}{\diff t}\bigg\vert_{t=0} \W_{p}^p(\WHeat_t\delta_x,\WHeat_t\delta_y)\\
&\qquad\qquad\leq \limsup_{t\downarrow 0} \frac{1}{t}\EE\bigg[\Big(1-\expo^{\int_0^{2t} p\kkkk\left(\B_{r}^{1},\B_{r}^{2}\right)/2\d r}\Big)\met^p\!\big(\B_{2t}^{1},\B_{2t}^{2}\big)\bigg]  + \frac{\diff^+}{\diff t}\bigg\vert_{t=0}W_{p'}^{\kk}(\delta_x,\delta_y,t)^p\\
&\qquad\qquad \leq -p\,\kkkk(x,y)\met^p(x,y) + \frac{\diff^+}{\diff t}\bigg\vert_{t=0} W_{p'}^{\kk}(\delta_x,\delta_y,t)^p
\end{align*}
for all $\kkkk\in\Cont_\bdd(\mms\times\mms)$ with $\kkkk\leq\kk$ on $\mms\times\mms$. Letting $p'\downarrow p$, the last upper derivative becomes nonpositive due to $\PTE_p(\k)$, and approximating $\kk$ from below using Lemma \ref{Le:kk approximation} gives the conclusion.
\end{proof}

Using this equivalence, Hölder's inequality and the chain rule, the subsequent nestedness of $\PTE_p(\k)$, which is the Lagrangian analogue of Lemma \ref{hierarchy}, is easily shown.

\begin{cor}\label{Nestedness TE_p} If $\PTE_p(\k)$ holds for some $p\in [1,\infty)$, then $\PTE_{p'}(\k)$ is satisfied for all $p'\in [1,p]$.
\end{cor}

\subsection{Transport estimates via vertical Brownian perturbations}\label{4.?}

We prove the variable Kuwada duality from Theorem \ref{gra-tra}. We start by first showing the implication from $\GE_q(\k)$ to $\PTE_p(\k)$, where $p,q\in (1,\infty)$ are dual to each other. Since the behavior of Brownian trajectories can only be controlled for small times, we show the equivalent infinitesimal first-order description of $\PTE_p(\k)$ in terms of a differential $p$-transport estimate. This is done by a localization argument.

Additionally, in the extremal case $q = 1$, the argument mentioned above can actually be circumvented and we are able to derive the contraction estimate
\begin{equation*}
\frac{\diff^+}{\diff t}\W_p^p(\WHeat_{t}\mu,\WHeat_{t}\nu) \leq -p\int_0^1\int_{\Geo(\mms)}\k(\gamma_s)\,\vert\dot\gamma\vert^p\d\boldsymbol{\pi}_t(\gamma)\d s\quad\text{for every }t\geq 0
\end{equation*}
for all $\mu,\nu\in\Prob(\mms)$ of finite $\W_p$-distance to each other, for \emph{every} $p\in (1,\infty)$. The measure $\boldsymbol{\pi}_t \in\Prob(\Geo(\mms))$ induces an \emph{arbitrary} $\W_p$-optimal coupling of $\WHeat_t\mu$ and $\WHeat_t\nu$. 
This is discussed now, see Theorem \ref{Th:Vertical Wasserstein contraction} and Corollary \ref{Cor:Derivative of W_p}, where, possibly replacing $\k$ by $\min\{\k,n\}$ for $n\in\N$, we assume that $\k$ is bounded. This is not restrictive as, given these results for every $n\in\N$, they easily pass to the limit $n\to\infty$ by monotone convergence.

Given $p\in (1,\infty)$ and $t\geq 0$, we define the function $\met_{p,\k,t}^0\colon \mms\times\mms\to \R$ by
\begin{equation*}
\met_{p,\k,t}^0(x,y) :=  \inf_{\gamma \in \G_0(x,y)}\bigg(\int_0^1 \Exp_{\gamma_s}\Big[\expo^{-\int_0^{2t}p\k(\B_{r})/2\d r}\Big]\,\vert\dot{\gamma}\vert^p\d s\bigg)^{1/p}.
\end{equation*}
Here $\smash{\big(\PP_{\gamma_s},\B\big)}$ denotes Brownian motion starting in $\gamma_s$ for every $s\in [0,1]$. We will not explicitly mention the dependence of the process $\B$ on $s$. The function $\met_{p,k,t}^0$ can be turned into a metric $\met_{p,k,t}$ on $\mms$ by defining
\begin{align*}
\met_{p,\k,t}(x,y) := \inf\!\Big\lbrace\sum_{i=1}^n \met_{p,\k,t}^0(x_{i-1}, x_i) : n \in \N,\ \! x=:x_0 < x_1 < \dots < x_n := y \Big\rbrace.
\end{align*}
It is equivalent to $\met$ by boundedness of $\k$ since $\met$ is a length metric. Let us denote by $\smash{\W_{p,\k,t}^0}$ and $\W_{p,k,t}$ the transport ``distances'' w.r.t.~$\met_{p,\k,t}^0$ and $\met_{p,\k,t}$, respectively. Then $\W_{p,\k,t}$ is a metric on $\Prob_p(\mms)$, which is equivalent to the usual $p$-Wasserstein metric $\W_p$. Compared to the perturbed $p$-transport cost $\smash{W_p^\k}$ which measures Brownian evolutions ``horizontally'' by following their trajectories with \emph{fixed} starting points, the distance $\W_{p,k,t}$ varies the initial points along a geodesic and may thus be seen as a ``vertical'' counterpart of $\W_p^\k$. 

Let $Q_s$ be the $p$-Hopf--Lax semigroup and $q\in (1,\infty)$ such that $1/p+1/q=1$. Similarly to~\cite[Proposition~3.7]{kuwada2010}, the key point will be the following Lipschitz regularity along geodesics.

\begin{lma}\label{Le:Curve inequality} Let $f\in \Lip_\bdd(\mms)$. Then for every $x,y\in\mms$ and all $\gamma \in \G_0(y,x)$, the map $s\mapsto \ChHeat_tQ_sf(\gamma_s)$ belongs to $\Lip([0,1])$, and
\begin{equation*}
\ChHeat_t Q_1 f(x) - \ChHeat_tf(y) \leq \int_0^1 \Big(\limsup_{h\downarrow 0} \frac{1}{h}\big(\ChHeat_tQ_sf(\gamma_{s+h}) - \ChHeat_tQ_sf(\gamma_s)\big) - \frac{1}{q}\ChHeat_t\big(\lip(Q_sf)^q\big)(\gamma_s)\Big) \d s.
\end{equation*}
\end{lma}

\begin{proof} Let $h>0$ and $s\in [0,1-h]$. Notice that
\begin{align*}
&\frac{1}{h}\big\vert\ChHeat_tQ_{s+h}f(\gamma_{s+h}) - \ChHeat_tQ_sf(\gamma_s)\big\vert\\
&\qquad\qquad \leq \frac{1}{h}\big\vert \ChHeat_tQ_{s+h}f(\gamma_{s+h}) - \ChHeat_tQ_{s+h}f(\gamma_s)\big\vert + \frac{1}{h}\Big\vert\! \int_\mms \big(Q_{s+h}f - Q_sf\big) \d\WHeat_t\delta_{\gamma_s}\Big\vert\\
&\qquad\qquad \leq \frac{\met(x,y)}{h}\int_s^{s+h} \vert\Diff\ChHeat_tQ_{s+h}f\vert(\gamma_v)\d v + \int_\mms\frac{1}{h}\vert Q_{s+h}f - Q_sf\vert \d\WHeat_t\delta_{\gamma_s}.
\end{align*}
The latter is bounded uniformly in $s$ and $h$ since the first integral can be controlled using the Lipschitz regularization estimate \eqref{Lip reg est} of the heat flow while the second one exploits the fact that the map $s\mapsto Q_s f$ is Lipschitz from $[0,\infty)$ to $\Cont(\mms)$. 

It follows that $\ChHeat_tQ_1f(x) - \ChHeat_tf(y)$ can be written as
\begin{equation}\label{I II decomposition}
\int_0^1 \Big(\limsup_{h\downarrow 0} \frac{1}{h}\big(\ChHeat_tQ_sf(\gamma_{s+h}) - \ChHeat_tQ_sf(\gamma_s)\big) + \limsup_{h\downarrow 0}\frac{1}{h}\int_\mms \big(Q_{s+h}f - Q_sf\big) \d \WHeat_t\delta_{\gamma_{s+h}}\Big)\d s.
\end{equation}
The Kantorovich--Rubinstein formula \eqref{Eq:Kantorovich} for $\W_1$, the $\W_1$-contractivity of the heat flow and the duality of $\ChHeat_t$ and $\WHeat_t$ give us the following upper bound for the second $\limsup$ in \eqref{I II decomposition}
\begin{align*} 
&\limsup_{h\downarrow 0} \frac{1}{h}\int_\mms \big(Q_{s+h}f - Q_sf\big) \d \big(\WHeat_t\delta_{\gamma_{s+h}} - \WHeat_t\delta_{\gamma_s}\big) + \limsup_{h\to 0}\frac{1}{h}\int_\mms\big(Q_{s+h}f - Q_sf\big) \d \WHeat_t\delta_{\gamma_s}\\
&\qquad\qquad\leq \Lip(Q_\bullet f)\, \limsup_{h\downarrow 0} \W_1\big(\WHeat_t\delta_{\gamma_{s+h}},\WHeat_t\delta_{\gamma_s}\big) + \int_\mms \frac{\diff}{\diff s} Q_sf\d \WHeat_t\delta_{\gamma_s}\\
&\qquad\qquad = -\frac{1}{q}\int_\mms \lip(Q_sf)^q\d \WHeat_t\delta_{\gamma_s} = -\frac{1}{q}\ChHeat_t\big(\lip(Q_sf)^q\big)(\gamma_s).
\end{align*}
Here we used $\Lip(Q_\bullet f)$ as a shorthand for the Lipschitz constant of the map $s\mapsto Q_s f$ from $[0,\infty)$ to $\Cont(\mms)$. These estimates conclude the proof.
\end{proof}

\begin{thm}\label{Th:Vertical Wasserstein contraction} Assume the $1$-gradient estimate $\GE_1(\k)$. Then for every $p\in (1,\infty)$, $t\geq 0$ and $\mu,\nu\in\Prob(\mms)$,
\begin{equation*}
\W_p(\WHeat_t\mu,\WHeat_t\nu) \leq \W_{p,\k,t}(\mu,\nu)\leq \W_{p,\k,t}^0(\mu,\nu).
\end{equation*}
\end{thm}

\begin{proof} Without loss of generality, we consider $\mu := \delta_x$ and $\nu := \delta_y$ for $x,y\in\mms$, and $t>0$ as the general case (or, to be more precise, the first inequality, since only $\met_{p,\k,t}$ is continuous in general) is covered by a standard coupling argument, see e.g.~\cite[Theorem 4.4]{savare2014} or \cite[Lemma 3.3]{kuwada2010}. It suffices to prove  $\smash{\W_p(\WHeat_t\mu,\WHeat_t\nu) \leq \W_{p,\k,t}^0(\mu,\nu) = \met_{p,\k,t}^0(x,y)}$ since the first claimed inequality then easily follows by definition of $\met_{p,\k,t}$, and by construction $\W_{p,\k,t}(\mu,\nu)\leq \W_{p,\k,t}^0(\mu,\nu)$.

By the duality \eqref{Eq:Kantorovich}, we have to estimate $\ChHeat_tQ_1f(x) - \ChHeat_tf(y)$ from above for every $f\in\Lip_\bdd(\mms)$. Pick a geodesic $\gamma \in \G_0(y,x)$. By the upper gradient property of $\vert\Diff \ChHeat_t Q_sf\vert$ and the $\GE_1(\k)$ inequality, we deduce for $\Leb^1$-a.e.~$s\in [0,1]$ that
\begin{align*}
\limsup_{h\downarrow 0} \frac{1}{h}\big(\ChHeat_tQ_sf(\gamma_{s+h}) - \ChHeat_tQ_sf(\gamma_s)\big)
&\leq \limsup_{h\downarrow 0}\frac{\met(x,y)}{h}\int_s^{s+h} \Schr^k_t\vert\Diff Q_sf\vert(\gamma_v)\d v\\
&\leq \met(x,y)\, \Exp_{\gamma_s}\Big[\expo^{-\int_0^{2t} pk(\B_{r})/2\d r}\Big]^{\!1/p}\, \ChHeat_t\big(\lip(Q_sf)^q\big)^{1/q}(\gamma_s),
\end{align*}
denoting by $\big(\PP_{\gamma_s},\B\big)$ Brownian motion on $\mms$ starting in $\gamma_s$. Invoking Lemma \ref{Le:Curve inequality} and Young's inequality, we infer that
\begin{equation*}
\ChHeat_tQ_1f(x) - \ChHeat_tf(y) \leq \frac{\met^p(x,y)}{p} \int_0^1 \Exp_{\gamma_s}\Big[\expo^{-\int_0^{2t}pk(\B_{r})/2\d r}\Big]\d s.
\end{equation*}
Taking the supremum over $f\in\Lip_\bdd(\mms)$ and then infimizing over all geodesics $\gamma$ connecting $y$ to $x$, we conclude the desired inequality.
\end{proof}

With this in hand, we can proceed to what we have indicated in Remark \ref{Re:WDTE}, i.e.~that actually, a much stronger assertion than just a control on the upper derivative of the function $t\mapsto \W_p^p(\WHeat_t\delta_x,\WHeat_t\delta_y)$ at $0$ is possible.

\begin{cor}\label{Cor:Derivative of W_p} Assume that $\GE_1(k)$ is satisfied. Let $\mu,\nu\in\Prob(\mms)$ so that $\W_p(\mu,\nu) < \infty$, let $t\geq 0$, and let $\boldsymbol{\pi}_t\in\Prob(\Geo(\mms))$ represent an arbitrary $\W_p$-optimal coupling between $\WHeat_t\mu$ and $\WHeat_t\nu$, i.e.~$(\eval_0,\eval_1)_\push\boldsymbol{\pi}_t$ is a $\W_p$-optimal coupling of $\WHeat_t\mu$ and $\WHeat_t\nu$. Then
\begin{equation*}
\frac{\diff^+}{\diff t}\W_p^p(\WHeat_{t}\mu,\WHeat_{t}\nu) \leq -\int_0^1\int_{\mathrm{Geo(\mms)}} k(\gamma_s)\,\vert\dot\gamma\vert^p\d\boldsymbol{\pi}_t(\gamma)\d s.
\end{equation*}
\end{cor}

\begin{proof} Given any optimal geodesic plan $\boldsymbol{\pi}_t$ as above, using Theorem \ref{Th:Vertical Wasserstein contraction} gives
\begin{align*}
&\limsup_{h\to 0}\frac{1}{ph}\left(\W_p^p(\WHeat_{t+h}\mu,\WHeat_{t+h}\nu) - \W_p^p(\WHeat_t\mu,\WHeat_t\nu)\right)\\
&\qquad\qquad \leq \limsup_{h\downarrow 0} \frac{1}{ph}\left(\W_{p,k,h}^0(\WHeat_t\mu,\WHeat_t\nu)^p - \W_p^p(\WHeat_t\mu,\WHeat_t\nu)\right)\\
&\qquad\qquad\leq \limsup_{h\downarrow 0} \frac{1}{ph} \int_{\Geo(\mms)} \bigg(\int_0^1 \Exp_{\gamma_s}\Big[\expo^{-\int_0^{2h}pk(\B_{r})/2\d r}\Big]\d s - 1\bigg) \met^p(\gamma_0,\gamma_1) \d\boldsymbol{\pi}_t(\gamma)\\
&\qquad\qquad = -\int_0^1\int_{\mathrm{Geo}(\mms)}  k(\gamma_s)\,\vert\dot\gamma\vert^p \d\boldsymbol{\pi}_t(\gamma)\d s,
\end{align*}
where $\big(\PP_{\gamma_s},\B\big)$ denotes Brownian motion on $\mms$ starting in $\gamma_s$. In the very last step, we used the assumed boundedness of $k$ together with the dominated convergence theorem. 
\end{proof}

\begin{remark}\label{sup improvement} The previous corollary applied to $\mu := \delta_x$ and $\nu := \delta_y$ for $x,y\in\mms$ at $t=0$, choosing $\boldsymbol{\pi}_0$ as the Dirac mass on an arbitrary geodesic $\gamma\in\G_0(x,y)$, yields the estimate
\begin{equation*}
\frac{\diff^+}{\diff t}\bigg\vert_{t=0} \W_p^p(\WHeat_t\delta_x,\WHeat_t\delta_y) \leq -p\,\sup_{\gamma\in\G_0(x,y)}\,\int_0^1\k(\gamma_s)\d s \met^p(x,y)\leq -p\, \kkkkk(x,y)\met^p(x,y),
\end{equation*}
where, as in \eqref{kkkkk def}, the function $\kkkkk\colon\mms\times\mms\to\R$ is defined by
\begin{equation*}
\kkkkk(x,y) := \liminf_{(x_n,y_n) \to (x,y)}\,\sup_{\gamma\in\G_0(x_n,y_n)}\,\int_0^1\k(\gamma_s)\d s.
\end{equation*}
Note that $\kkkkk$ is lower semicontinuous and bounded from below.

This improves the differential $p$-transport estimate \eqref{Eq:WDTE assumption}, since $\kk\leq \kkkkk$ on $\mms\times\mms$, see also Proposition \ref{Pr:EVI implies DTE}. In Chapter \ref{Ch:Pathwise coupling}, we shall construct a coupling of Brownian motions obeying pathwise bounds involving the larger function $\kkkkk$ in place of $\kk$. In particular, using Theorem \ref{propn:transport2gradient 1}, all equivalences from Theorem \ref{Th:Huge thm} and Theorem \ref{gra-tra} are still valid when replacing the function $\kk$ by $\kkkkk$ in all relevant quantities.
\end{remark}

The proof of the $\PTE_p(\k)$ property starting from $\GE_q(\k)$ with dual $p,q\in (1,\infty)$ is slightly more involved as a control of the error terms is only possible ``locally'' for small times. A crucial ingredient is the subsequent result.

\begin{lma}\label{Le:Infinitesimal estimate} Let $u$ and $v$ be bounded Borel functions on $\mms$ such that $u\leq v$ on a ball $B_\delta(z)$, $z\in\mms$ and $\delta > 0$. Then for every $p \in (1,\infty)$ and $\varepsilon > 0$, there exists $t_* > 0$ such that for every $t\in [0,t_*]$, every nonnegative Borel function $g$ on $\mms$, and every Brownian motion $(\PP_x, \B)$ on $\mms$ starting in $x\in B_{\delta/2}(z)$, we have
\begin{equation*}
\EE_x\Big[\expo^{\int_0^t u(\B_r)\d r}g(\B_t)\Big] \leq \EE_x\Big[\expo^{p\int_0^t (v(\B_r)+\varepsilon)\d r}g^p(\B_t)\Big]^{1/p}.
\end{equation*}
\end{lma}

\begin{proof} The condition on $u$ and $v$ guarantees that for fixed $T>0$ and every $t\in[0,T]$,
\begin{equation*}
\expo^{\int_0^t u(\B_r)\d r} - \expo^{\int_0^t v(\B_r)\d r} = \int_0^t\expo^{\int_0^s u(\B_r)\d r + \int_s^t v(\B_r)\d r}(u-v)(\B_s) \d s \leq M\int_0^t \One_{\{\B_s\notin B_{\delta}(z)\}}\d s.
\end{equation*}
Here, $M>0$ is a constant depending only on $u$, $v$ and $T$. Therefore,
\begin{align*}
\EE_x\Big[\expo^{\int_0^t u(\B_r)\d r} g(\B_t)\Big] &\leq \EE_x\Big[\expo^{\int_0^t v(\B_r)\d r} g(\B_t)\Big] + M\int_0^t \EE_x\Big[\expo^{\int_0^t v(\B_r)\d r}g(\B_t) \One_{\{\B_s\notin B_{\delta}(z)\}}\Big]\d s\\
&\leq \EE_x\Big[\expo^{\int_0^t pv(\B_r)\d r}g^p(\B_t)\Big]^{1/p} \, \Big(1 + M \int_0^t \PP_x\big[\B_s \notin B_\delta(z)\big]^{1/q}\d s\Big),
\end{align*}
where $q\in (1,\infty)$ denotes the dual exponent to $p$. By Lemma \ref{tail estimate}, we know that $\PP_x[\B_s \notin B_\delta(z)] \leq s^q$ for every $s\in [0,t]$ and small enough $t$. Thus, $1+M \smash{\int_0^t \PP_x[\B_s \notin B_\delta(z)]^{1/q}\d s \leq \expo^{\varepsilon t}}$, which directly proves the claim.
\end{proof}

\begin{remark} With the very same strategy, also estimates for Feynman--Kac-type expressions in terms of pairs of Brownian motions can be derived, each component being required to start within $B_{\delta/2}(z)$. Moreover, the integrands $u$ and $v$ are then supposed to be functions on $\mms\times\mms$ with $u\leq v$ on $B_\delta(z)\times B_\delta(z)$.
\end{remark}

\begin{Prop}\label{Pr:Local contraction} Let $p,q\in (1,\infty)$ such that $1/p+1/q = 1$ and assume the $q$-gradient estimate $\GE_q(\k)$. Assume that $\kkkk\in\Cont_\bdd(\mms\times\mms)$ with $\kkkk\leq \kk$ on $\mms\times\mms$, and put $\kkk(x) := \kkkk(x,x)$ for $x\in\mms$. Then for every $\varepsilon > 0$, $p' \in (1,p)$ and $z\in \mms$, there exist $\delta > 0$ and $t_* > 0$ such that for every $x,y\in B_\delta(z)$, every $\gamma\in\G_0(y,x)$ and every $t\in [0,t^*]$, we have
\begin{equation*}
\W_{p'}^{p'}(\WHeat_t\delta_x,\WHeat_t\delta_y) \leq \met(x,y)\, \expo^{-\big(\int_0^1\kkk(\gamma_r)\d r - \varepsilon\big)t},
\end{equation*}
and thus in particular,
\begin{equation*}
\frac{\diff^+}{\diff t}\bigg\vert_{t=0} \W_{p'}(\WHeat_t\delta_x,\WHeat_t\delta_y) \leq -\met(x,y)\, \Big(\int_0^1\kkk(\gamma_r)\d r-\varepsilon\Big).
\end{equation*}
\end{Prop}

\begin{proof} We adapt the proof of Theorem \ref{Th:Vertical Wasserstein contraction} by adding a localization argument. Given $z\in\mms$ and $\varepsilon > 0$, choose $\delta >0$ and $L_z \in \R$ such that $L_z \leq \kkk \leq L_z+\varepsilon/2$ on $B_{3\delta}(z)$. Let $x,y \in B_{\delta}(z)$ and $\gamma \in\G_0(y,x)$, and note that $\smash{L_z \leq \int_0^1 \kkk(\gamma_r)\d r \leq L_z +\varepsilon/2}$.

Denote by $Q_s$ the $p'$-Hopf--Lax semigroup with dual exponent $q' \in (q,\infty)$. Since $\vert\Diff\ChHeat_t Q_sf\vert$ is a weak upper gradient and using $\GE_q(\k)$, which clearly implies $\GE_q(\kkk)$, we directly obtain, for $\Leb^1$-a.e.~$s\in[0,1]$,
\begin{align*}
\limsup_{h\downarrow 0} \frac{1}{h}\big(\ChHeat_tQ_sf(\gamma_{s+h}) - \ChHeat_tQ_sf(\gamma_s)\big)\leq \met(x,y)\, \big(\Schr^{q\kkk}_t\vert\Diff Q_sf\vert^q\big)^{1/q}(\gamma_s).
\end{align*}
Applying Lemma \ref{Le:Infinitesimal estimate} with $\varepsilon/2$ and $t/2$ in place of $\varepsilon$ and $t$, respectively, we get, for small enough $t$,
\begin{equation*}
\big(\Schr^{q\kkk}_t\vert\Diff Q_sf\vert^q\big)^{1/q}(\gamma_s) \leq \expo^{-(L_z-\varepsilon/2)t}\, \ChHeat_t\big(\lip(Q_sf)^{q'}\big)^{1/q'}(\gamma_s),
\end{equation*}
and thus
\begin{equation*}
\met(x,y)\, \big(\Schr^{q\kkk}_t\vert\Diff Q_sf\vert\big)^{1/q}(\gamma_s) \leq \frac{\met^{p'}\!(x,y)}{p'}\,\expo^{-p'(L_z-\varepsilon/2)t} + \frac{1}{q'}\ChHeat_t\big(\lip(Q_sf)^{q'}\big)(\gamma_s)
\end{equation*}
for $\Leb^1$-a.e.~$s\in[0,1]$ by Young's inequality. Therefore, Lemma \ref{Le:Curve inequality} with $q'$ in place of $q$ yields
\begin{equation*}
\ChHeat_tQ_1f(x) - \ChHeat_tf(y) \leq \frac{\met^{p'}\!(x,y)}{p'}\,\expo^{-p'(L_z-\varepsilon/2)t}\leq \frac{\met^{p'}\!(x,y)}{p'}\,\expo^{-p'\big(\int_0^1\kkk(\gamma_r)\d r - \varepsilon\big)t}.
\end{equation*}
Taking the supremum over $f\in\Lip_\bdd(\mms)$, we conclude by \eqref{Eq:Kantorovich}.
\end{proof}

\begin{thm}\label{Th: Kuwada GEq to TEp} Given $p,q\in (1,\infty)$ with $1/p+1/q=1$, the $q$-gradient estimate $\GE_q(\k)$ implies the $p$-transport estimate $\PTE_p(\k)$.
\end{thm}

\begin{proof} Fix $x,y\in\mms$, an arbitrary geodesic $\gamma \in \G_0(y,x)$ and $\kkk$ as in Proposition \ref{Pr:Local contraction}. Given $\varepsilon > 0$, choose a finite covering of $\gamma([0,1])$ by metric balls $B_{\delta_i/2}(\gamma_{s_i})$, $i \in\{1,\dots,n\}$ and $n\in\N$, such that each of the enlarged balls $B_{\delta_i}(\gamma_{s_i})$ satisfies the assumption of the previous Proposition \ref{Pr:Local contraction}. Without restriction, we may assume $s_1 = 0$ and $s_n = 1$. Applying this proposition to pairs of intermediate points $\gamma_{s_{i-1}}$ and $\gamma_{s_i}$ and the reparameterized geodesics $\smash{\gamma^i\in \G_0(\gamma_{s_{i-1}},\gamma_{s_i})}$ defined by $\smash{\gamma^i_r := \gamma_{s_{i-1} + r(s_i - s_{i-1})}}$, $r\in [0,1]$, yields
\begin{align*}
\frac{\diff^+}{\diff t}\bigg\vert_{t=0} \W_{p'}(\WHeat_t\delta_x,\WHeat_t\delta_y) &\leq \sum_{i=1}^n \frac{\diff^+}{\diff t}\bigg\vert_{t=0} \W_{p'}\big(\WHeat_t\delta_{\gamma_{s_{i-1}}}, \WHeat_t\delta_{\gamma_{s_i}}\big)\\
&\leq -\sum_{i=1}^n \met(\gamma_{s_{i-1}},\gamma_{s_i})\, \Big(\int_0^1 \kkk\big(\gamma^i_r\big)\d r - \varepsilon\Big)\\
&= -\met(x,y)\,\Big(\int_0^1\kkk(\gamma_r)\d r - \varepsilon\Big).
\end{align*}
Since $\kkk$ is arbitrary, this bound holds with $\k$ in place of $\kkk$ by Lemma \ref{Le:kk approximation}. Furthermore, by definition of $\kk$ and the arbitrariness of $\varepsilon > 0$, we deduce the differential transport estimate \eqref{Eq:WDTE assumption} with $p$ replaced by $p'$. Since this true for every $p' \in (1,p)$, this finally yields $\PTE_p(\k)$ by Theorem \ref{thm:nonincreasing} and monotone convergence.
\end{proof}

\subsection{Gradient estimates out of pathwise and transport estimates}\label{Sec:4.4}

A modification of the arguments given in \cite[Proposition 3.1]{kuwada2010} allows us to prove the converse direction of Theorem \ref{gra-tra}, i.e.~that the $p$-transport estimate $\PTE_p(\k)$ implies the $q$-gradient estimate $\GE_q(\k)$, where $1/p+1/q = 1$. As in the previous section, a control of the error terms can only be achieved for small times. Therefore, instead of deriving $\GE_q(\k)$ directly, it is more convenient to establish a local version of the $q$-Bochner inequality $\BE_q(\k,\infty)$. 

As in the preceding Section \ref{4.?}, the extremal version $q=1$ is much easier to treat: in this case, the condition ``$\PTE_\infty(\k)$'' is to be interpreted as ``$\PTE_p(\k)$ holds for any $p\in [1,\infty)$'', which translates into the requirement of $\PCP(\k)$ as discussed in Chapter \ref{Ch:Pathwise coupling}.

\begin{thm}\label{propn:transport2gradient 1}
The property $\PCP(k)$ implies the 1-gradient estimate $\GE_1(\k)$, that is, for every $f\in\Sob^{1,2}(\mms)$ and $t\geq 0$, we have 
\begin{equation*}
\Gamma(\ChHeat_t f)^{1/2} \leq \Schr^{k}_t\big(\Gamma(f)^{1/2}\big)\quad\meas\text{-a.e.}
\end{equation*}
\end{thm}

\begin{proof} Fix $f\in \Lip_\bs(\mms)$ and $x\in\mms$. Recall that $\ChHeat_{t/2} f(x) = \EE_x[f(\B_t)]$, where $(\PP_x,\B)$ denotes Brownian motion on $\mms$ starting in $x$. Pick a function $\smash{\kkkk\in\Lip_\bdd(\mms\times\mms)}$ with $\kkkk\leq \kk$ on $\mms\times\mms$, and set $\kkk(x) := \kkkk(x,x)$ for $x\in\mms$. By $\PCP(\k)$, given any $\varrho  >0$ and $y\in B_\varrho(x)$, we may choose a pair $\smash{\big(\PP_{x,y},\B^1\big)}$ and $\smash{\big(\PP_{x,y},\B^2\big)}$ of coupled Brownian motions in such a way that $\PP_{x,y}$-a.s.,  we have
\begin{equation}\label{PCP assumption}
\met\!\big(\B_t^1,\B_t^2\big) \leq \expo^{-\int_0^t \kk\left(\B_r^1,\B_r^2\right)/2\d r} \met(x,y) \leq \expo^{-\int_0^t \kkkk\left(\B_r^1,\B_r^2\right)/2\d r} \met(x,y)
\end{equation} 
for every $t\geq 0$. With this in hand, we can estimate
\begin{align*}
\vert\Diff \ChHeat_{t/2}f\vert(x) &\leq \lim_{\varrho\downarrow 0} \sup_{y\in B_{\varrho}(x)} \frac{\vert \ChHeat_{t/2}f(x) - \ChHeat_{t/2}f(y)\vert}{\met(x,y)}\\
&\leq \lim_{\varrho\downarrow 0} \sup_{y\in B_{\varrho}(x)} \EE_{x,y}\bigg[\frac{\vert f(\B_t^1) - f(\B_t^2)\vert}{\met(\B_t^1,\B_t^2)} \, \frac{\met(\B_t^1,\B_t^2)}{\met(x,y)} \, \left(\One_{U_{\varrho,t}} + \One_{V_{\varrho,t}} + \One_{W_{\varrho,t}}\right)\bigg],
\end{align*}
where $\smash{V_{\varrho,t} := \big\lbrace\!\met\!\big(\B_t^1,\B_t^2\big)\geq \varrho^{1/2}\big\rbrace}$, $\smash{W_{\varrho,t} := \big\lbrace \!\int_0^t \met\!\big(\B_r^1,\B_r^2\big)\d r/t \geq \varrho^{1/2} \big\rbrace}$ and $U_{\varrho,t} := V^{\mathrm{c}}_{\varrho,t} \cap W^{\mathrm{c}}_{\varrho,t}$.

Let us consider this upper bound for the weak upper gradient $\vert\Diff\ChHeat_{t/2}f\vert(x)$ term by term, starting with the contribution coming from $U_{\varrho,t}$. We have the inequality $\smash{\int_0^t \kkkk\big(\B_r^1,\B_r^2\big)\d r}\geq \smash{\int_0^t\kkk\big(\B_r^1\big)\d r} - \smash{\Lip\big(\kkkk\big)t\varrho^{1/2}}$ on $W_{\varrho,t}^\mathrm{c}$, which gives
\begin{align*}
&\lim_{\varrho\downarrow 0} \sup_{y\in B_{\varrho}(x)} \EE_{x,y}\bigg[\frac{\vert f(\B_t^1) - f(\B_t^2)\vert}{\met(\B_t^1,\B_t^2)} \, \frac{\met(\B_t^1,\B_t^2)}{\met(x,y)} \, \One_{U_{\varrho,t}}\bigg]\\
&\qquad\qquad \leq \lim_{\varrho\downarrow 0} \sup_{y\in B_{\varrho}(x)}\EE_{x,y}\bigg[\expo^{-\int_0^t\kkk\left(\B_r^1\right)/2\d r + \Lip(\kkkk)t\varrho^{1/2}/2} \sup_{z\in B_{\varrho^{1/2}}(\B_t^1)} \Big\vert \frac{f(\B_t^1) - f(z)}{\met(\B_t^1, z)}\Big\vert\bigg]\\
&\qquad\qquad = \lim_{\varrho\downarrow 0} \EE_x\bigg[\expo^{-\int_0^t\kkk(\X_r)/2\d r + \Lip(\kkkk)t\varrho^{1/2}/2}\, \sup_{z\in B_{\varrho^{1/2}}(\X_t)} \Big\vert \frac{f(\X_t) - f(z)}{\met(\X_t,z)}\Big\vert\bigg]\\
&\qquad\qquad = \EE_x\Big[\expo^{-\int_0^t\kkk(\X_r)/2\d r} \, \vert\Diff f\vert(\X_t)\Big] = \Schr_{t/2}^{\kkk}\big(\Gamma(f)^{1/2}\big)(x).
\end{align*}
We point out the intermediate change from the process $\B^1$, which in general also depends on $y$, to a Brownian motion $\smash{\big(\PP_x,\X\big)}$ on $\mms$ starting in $x$, chosen independently of $y$.

Next we consider the term involving $\One_{V_{\varrho,t}}$. Denoting by $C>0$ a suitable upper bound on $\kkkk$, we obtain by \eqref{PCP assumption} that 
\begin{align*}
&\lim_{\varrho\downarrow 0} \sup_{y\in B_{\varrho}(x)} \EE_{x,y}\bigg[\frac{\vert f(\B_t^1) - f(\B_t^2)\vert}{\met(\B_t^1,\B_t^2)} \, \frac{\met(\B_t^1,\B_t^2)}{\met(x,y)} \, \One_{V_{\varrho,t}}\bigg]\\
&\qquad\qquad \leq \Lip(f)\, \lim_{\varrho\downarrow 0} \frac{1}{\varrho^{1/2}}\sup_{y\in B_{\varrho}(x)}\EE_{x,y}\bigg[\frac{\met^2(\B_t^1,\B_t^2)}{\met(x,y)}\bigg] \leq \Lip(f)\,\expo^{Ct}\, \lim_{\varrho\downarrow 0} \frac{1}{\varrho^{1/2}}\sup_{y\in B_{\varrho}(x)}\met(x,y) = 0.
\end{align*}
Similarly, the last expression which involves $W_{\varrho,t}$ can be bounded through
\begin{align*}
&\lim_{\varrho\downarrow 0}\sup_{y\in B_{\varrho}(x)} \EE_{x,y}\bigg[\frac{\vert f(\B_t^1) - f(\B_t^2)\vert}{\met(\B_t^1,\B_t^2)} \, \frac{\met(\B_t^1,\B_t^2)}{\met(x,y)} \, \One_{W_{\varrho,t}}\bigg]\\
&\qquad\qquad\leq\Lip(f) \, \lim_{\varrho\downarrow 0} \frac{1}{t\varrho^{1/2}} \sup_{y\in B_{\varrho}(x)} \int_0^t \EE_{x,y}\bigg[\frac{\met(\B_t^1,\B_t^2) \met(\B_r^1,\B_r^2)}{\met(x,y)}\bigg]\d r\\
&\qquad\qquad \leq\Lip(f)\, \expo^{Ct}\, \lim_{\varrho\downarrow 0} \frac{1}{\varrho^{1/2}} \sup_{y\in B_{\varrho}(x)}\met(x,y) = 0.
\end{align*}

Finally, we have to extend the class of admissible functions $f$ and pass to $\GE_1(\k)$. Every $f\in\Sob^{1,2}(\mms)$ can be approximated strongly in $\Sob^{1,2}(\mms)$ by a sequence of Lipschitz functions $f_n$ with bounded support. Since $\Gamma$ is quadratic, we have $\Gamma(f_n-f)\to 0$ in $\Ell^1(\mms,\meas)$ and thus, possibly passing to a subsequence, we get, for some suitable $c\in\R$, that
 \begin{equation*}
  \lim_{n\to \infty}  \Schr^{\kkk}_t\big(\Gamma(f-f_n)^{1/2}\big) \le \expo^{ct} \, \lim_{n\to\infty} \ChHeat_t\big(\Gamma(f-f_n)^{1/2}\big) = 0\quad\meas\text{-a.e.}
  \end{equation*}
Moreover, $\Gamma(\ChHeat_tf_n)\to \Gamma(\ChHeat_tf)$ in $\Ell^1(\mms,\meas)$ as $n\to\infty$ and thus, up to a subsequence, this convergence holds $\meas$-a.e., which then proves $\GE_1(\kkk)$ for arbitrary $f\in\Sob^{1,2}(\mms)$. By the arbitrariness of $\kkkk$, Lemma \ref{Le:kk approximation} and the identity $\k(x) = \kk(x,x)$ for every $x\in\mms$, we deduce $\GE_1(\k)$ by the monotone convergence theorem.
\end{proof}

\begin{Prop}\label{Pr:Local duality} Let $\varepsilon > 0$, $z \in\mms$ and $q\in (1,\infty)$. Assume the transport estimate $\PTE_p(\k)$, where $1/p + 1/q = 1$. Suppose that $\kkkk\in\Cont_\bdd(\mms\times\mms)$ with $\kkkk\leq \kk$ on $\mms\times\mms$. Then for every $q' \in (q,\infty)$, there exist $t_* > 0$ and $\delta > 0$ such that
\begin{equation*}
\Gamma(\ChHeat_t f)^{q'/2} \leq \Schr_t^{q'(\kkk-\varepsilon)}\big(\Gamma(f)^{q'/2}\big)\quad\meas\text{-a.e. on~}B_\delta(z)
\end{equation*}
for every $t\in [0,t_*]$ and all bounded Lipschitz functions $f$ on $\mms$.
\end{Prop}

\begin{proof} Fix $T>0$. Given $\varepsilon > 0$, choose $\delta > 0$ and $L_z\in \R$ such that $L_z \leq \kkkk(x,y) \leq L_z + \varepsilon/3$ for every $x,y\in B_{3\delta}(z)$. Given $t\in [0,T]$, $x\in B_\delta(z)$ and $y \in B_\varrho(z)$ with $\varrho\leq \delta$, select a pair $\smash{\big(\PP_{x,y},\B^1\big)}$ and $\smash{\big(\PP_{x,y},\B^2\big)}$ of coupled Brownian motions starting in $(x,y)$ which attains the minimum in the definition of $\smash{W_p^\kk(\delta_x,\delta_y,t/2)} \leq \met(x,y)$. The choice of this pair does depend on $x$, $y$ and $t$, but these dependencies are suppressed in the notation. Similarly to the proof of Theorem \ref{propn:transport2gradient 1}, for every $f\in\Lip_\bdd(\mms)$, we have
\begin{equation*}
\vert\Diff\ChHeat_{t/2}f\vert(x) \leq \lim_{\varrho\downarrow 0} \sup_{y\in B_{\varrho}(x)} \EE_{x,y}\bigg[\frac{\vert f(\B_t^1) - f(\B_t^2)\vert}{\met(\B_t^1,\B_t^2)} \, \frac{\met(\B_t^1,\B_t^2)}{\met(x,y)} \, \big(\One_{V_{\varrho,t}} + \One_{V_{\varrho,t}^\mathrm{c}}\big)\bigg]
\end{equation*}
where $\smash{V_{\varrho,t} := \big\lbrace \!\met\!\big(\B_t^1,\B_t^2\big) \geq \varrho^{1/2q}\big\rbrace}$. The contribution of $V_{\varrho,t}$ vanishes as $\varrho\downarrow 0$ due to
\begin{align*}
&\lim_{\varrho\downarrow 0}\sup_{y\in B_{\varrho}(x)} \EE_{x,y}\bigg[\frac{\vert f(\B_t^1) - f(\B_t^2)\vert}{\met(x,y)} \,\One_{V_{\varrho,t}}\bigg]\\
&\qquad\qquad \leq \Lip(f)\,\expo^{Ct}\, \lim_{\varrho\downarrow 0} \varrho^{(1-p)/2q}\, \sup_{y\in B_{\varrho}(x)} \frac{1}{\met(x,y)}\, \EE_{x,y}\Big[\expo^{\int_0^t p\kk\left(\B_r^1,\B_r^2\right)/2 \d r}\met^p\!\big(\B_t^1,\B_t^2\big)\Big]\\
&\qquad\qquad \leq\Lip(f)\, \expo^{Ct} \,\lim_{\varrho\downarrow 0} \varrho^{(1-p)/2q}\,\sup_{y\in B_\varrho(x)} \met^{p-1}(x,y) = 0
\end{align*}
for a suitable $C>0$, where we used the assumption  that $\kkkk\leq \kk$ in the first inequality and the $\PTE_p(\k)$ condition in the last inequality.

Next we study the influence coming from $V_{\varrho,t}^\mathrm{c}$.  Choosing some exponents $q''\in (q,q')$ and $p'' \in (1,p')$ dual to each other, using Hölder's inequality, Lemma \ref{Le:Infinitesimal estimate} with $\varepsilon/3$ and $t/2$ in place of $\varepsilon$ and $t$, respectively, and eventually assumption $\PTE_p(\k)$, we obtain for sufficiently small $t$ that
\begin{align*}
&\lim_{\varrho\downarrow 0}\sup_{y\in B_\varrho(x)}\EE_{x,y}\bigg[\frac{\vert f(\B_t^1)-f(\B_t^2)\vert}{\met(\B_t^1,\B_t^2)} \, \frac{\met(\B_t^1,\B_t^2)}{\met(x,y)} \, \One_{V_{\varrho,t}^\mathrm{c}}\bigg]\\
&\qquad\qquad \leq \expo^{-(L_z-\varepsilon/3) t/2} \, \lim_{\varrho\downarrow 0} \sup_{y\in B_\varrho(x)} \EE_{x,y}\bigg[\Big\vert\frac{f(\B_t^1) - f(\B_t^2)}{\met(\B_t^1,\B_t^2)} \Big\vert^{q''} \, \One_{V_{\varrho,t}^\mathrm{c}}\bigg]^{1/q''}\\
&\qquad\qquad\qquad\qquad \cdot \lim_{\varrho\downarrow 0} \sup_{y\in B_\varrho(x)}\EE_{x,y}\bigg[\expo^{p''(L_z-\varepsilon/3)t/2}\, \Big\vert\frac{\met(\B_t^1,\B_t^2)}{\met(x,y)}\Big\vert^{p''}\bigg]^{1/p''}\\
&\qquad\qquad \leq \expo^{-(L_z-\varepsilon/3)t/2}\, \lim_{\varrho\downarrow 0}\EE_x\bigg[\sup_{z\in B_{\varrho^{1/2q}}(\X_t)} \Big\vert \frac{f(\X_t) - f(z)}{\met(\X_t,z)}\Big\vert^{q''}\bigg]^{1/q''}\,\frac{1}{\met(x,y)}\,\W_p^\kk(\delta_x,\delta_y,t)\\
&\qquad\qquad\leq \expo^{-(L_z -\varepsilon/3)t/2}\,\EE_x\big[\vert\Diff f\vert^{q''}\!(\X_t)\big]^{1/q''}.
\end{align*}
Here $(\PP_x,\X)$ is a Brownian motion on $\mms$ starting in $x$ which is chosen independently of $y$. Once again using Lemma \ref{Le:Infinitesimal estimate} as above to estimate the last expression, we finally obtain
\begin{equation*}
\lim_{\varrho\downarrow 0} \sup_{y\in B_\varrho(x)} \EE_{x,y}\bigg[\frac{\vert f(\B_t^1) - f(\B_t^2)\vert}{\met(\B_t^1,\B_t^2)}\, \frac{\met(\B_t^1,\B_t^2)}{\met(x,y)}\,\One_{V_{\varrho,t}^{\mathrm{c}}}\bigg] \leq \Schr_t^{q'(\kkk-\varepsilon)}\big(\vert\Diff f\vert^{q'}\big)^{1/q'}\!(x).\qedhere
\end{equation*}
\end{proof}

\begin{thm}\label{Th: Kuwada TEp to GEq} Given $p,q\in (1,\infty)$ with $1/p+1/q=1$, the $p$-transport estimate $\PTE_p(\k)$ implies the $q$-gradient estimate $\GE_q(\k)$.
\end{thm}

\begin{proof} Let $\kkkk$ be as in Proposition \ref{Pr:Local duality} and put $\kkk(x) := \kkkk(x,x)$ for $x\in\mms$. First, we assume that $q\in[2,\infty)$. Given $\varepsilon > 0$, $z\in \mms$, $t_* > 0$, $q'\in (q,\infty)$ and the associated time $t_* > 0$ from in Proposition \ref{Pr:Local duality}, arguing as in the proof of Theorem \ref{Th:BEq equiv GEq}, the function $F\colon [0,t_*]\to\R$ defined by
\begin{equation*}
F(t) := \int_{\mms} \Big(\Schr_t^{q'(\kkk-\varepsilon)}\big(\Gamma(f)^{q'/2}\big) - \Gamma(\ChHeat_tf)^{q'/2}\Big)\phi\d \meas
\end{equation*}
belongs to $\Cont^1([0,t_*])$ for every $f\in\TestF(\mms)$ and all nonnegative functions $\phi \in \Sob^{1,2}(\mms)\cap\Ell^\infty(\mms,\meas)$ supported in $B_\delta(z)$. The function $F$ itself and its derivative at $0$ are nonnegative by Proposition \ref{Pr:Local duality}. The latter translates into
\begin{equation*}
-\int_\mms \Big(\frac{1}{q'}\Gamma\big(\Gamma(f)^{q'/2},\phi\big) + \Gamma(f)^{q'/2}\,\Gamma(f,\Delta f)\,\phi\Big)\d\meas \geq \int_\mms (\kkk-\varepsilon)\, \Gamma(f)^{q'/2}\,\phi\d\meas.
\end{equation*}
Approximating $\k$ from below by the sequence $\k_n\in\Lip_\bdd(\mms)$ of functions $\k_n(x) := \kk_n(x,x)$ for $x\in\mms$, or in other words, replacing $\kkkk$ by $\kk_n$ for every $n\in\N$, where $\kk_n$ tends to $\kk$ from below as provided by Lemma \ref{Le:kk approximation}, and letting $q' \downarrow q$ and $\varepsilon \downarrow 0$, we obtain precisely the local $q$-Bakry--Émery inequality $\BE_{q,\loc}(\k,\infty)$ according to Definition \ref{Def:Local BE}. Since the latter implies $\BE_q(\k,\infty)$ by Theorem \ref{Thm:BE local global}, the equivalence with $\GE_q(\k)$ finishes the proof in the case $q\in[2,\infty)$.

If $q\in [1,2)$, choosing $q' := 2$ in Proposition \ref{Pr:Local duality} and arguing as above, we obtain $\BE_2(\k,\infty)$, which in turn implies $\BE_q(\k,\infty)$.
\end{proof}

\section{A pathwise coupling estimate}\label{Ch:Pathwise coupling}
This section is dedicated to the proof of the existence of a pair $\smash{\big(\PP,\B^1\big)}$ and $\smash{\big(\PP,\B^2\big)}$ of coupled Brownian motions with arbitrary initial distributions, under a slightly stronger assumption than $\PTE_p(\k)$ for large enough $p$, such that $\PP$-a.s.,
\begin{equation*}
\met\!\big(\B_{t}^1,\B_{t}^2\big) \leq \expo^{-\int_s^t \kkkkk\left(\B_r^1,\B_r^2\right)/2\d r} \met\!\big(\B_s^1, \B_s^2\big)\quad\text{for every }s,t\in [0,\infty)\text{ with }s \leq t,
\end{equation*}
where $\kkkkk$ is defined as in \eqref{kkkkk def}. It is necessary to adapt the arguments from \cite[Section 2]{sturm2015} in a nontrivial way, since this pathwise estimate requires control of the entire path of $\smash{\big(\B^1,\B^2\big)}$ on the interval $[s,t]$ and not just at the endpoints.

\begin{thm}\label{pte to pcp} Suppose that, for all large enough $p\in (1,\infty)$, the map $\smash{t\mapsto \W_p^\kkkkk(\delta_x,\delta_y,t)}$ is nonincreasing on $[0,\infty)$ for every $x,y\in\mms$. Then for every $\mu_1,\mu_2\in\Prob(\mms)$ there exists a pair
$\smash{\big(\PP,\B^1)}$ and $\smash{\big(\PP,\B^2\big)}$ of coupled Brownian motions on $\mms$ with initial distributions $\mu_1$ and $\mu_2$, respectively, such that $\PP$-a.s., we have
\begin{equation*}
\met\!\big(\B_t^1,\B_t^2\big) \leq \expo^{-\int_s^t\kkkkk\left(\B_r^1,\B_r^2\right)/2\d r}\met\!\big(\B_s^1,\B_s^2\big)\quad\text{for every }s,t\in [0,\infty)\text{ with }s\leq t.
\end{equation*}
In particular, the pathwise coupling property $\PCP(\k)$ holds.
\end{thm}

The assumption of Theorem \ref{pte to pcp} above is satisfied if $\GE_1(\k)$ holds by Remark \ref{sup improvement}, and it implies $\PTE_p(\k)$ for all large enough $p\in (1,\infty)$ by the discussion from Theorem \ref{thm:nonincreasing} and Corollary \ref{Cor:TE ==> WDTE}. By nestedness of $p$-transport estimates from Corollary \ref{Nestedness TE_p}, we may suppose without restriction that the assumption of Theorem \ref{pte to pcp} holds for every $p\in (1,\infty)$.

The proof of this theorem will be subdivided into multiple steps. Firstly, we construct a coupled process starting in $\delta_x\otimes\delta_y$, $x,y\in\mms$, satisfying the desired pathwise contraction estimate on the interval $[0,1]$. Secondly, a gluing procedure will let us extend the process to  $[0,\infty)$. Finally, we use a coupling technique to allow for arbitrary initial distributions.
 
\subsection[Deterministic initial distributions and time interval {$[0,1]$}]{Deterministic initial distributions and time interval {\boldmath{$[0,1]$}}}

\begin{Prop}\label{Pr:Law at t} For every $t \geq 0$, there exists a universally measurable map
\begin{equation*}
\boldsymbol{\mu}^t\colon\quad\mms\times\mms\quad\longrightarrow\quad \Prob(\PS{t})
\end{equation*}
such that for every $x,y \in X$, the marginals of $\smash{\boldsymbol{\mu}_{x,y}^t := \boldsymbol{\mu}^t(x,y)}$ are laws of Brownian motions, restricted to $[0,t]$, starting in $x$ and $y$, respectively, and
\begin{equation*}
\met\!\big(\gamma_t^1,\gamma_t^2\big) \leq \expo^{-\int_0^t \kkkkk\left(\gamma_r^1,\gamma_r^2\right)/2\d r}\met(x,y)\quad\text{for }\boldsymbol{\mu}_{x,y}^t\text{-a.e.~}\gamma\in\PS{t}.
\end{equation*}
\end{Prop}

\begin{proof} Given $x,y \in \mms$ and an increasing sequence $(p_n)_{n\in\N}$ tending to $\infty$, denote by $\smash{\boldsymbol{\eta}_{x,y}^{t,n}} \in \Prob(\PS{t})$ the measure obtained by Lemma \ref{kernel} for the exponent $p_n$, $\kkkk$ replaced by $\kkkkk$, and time $t/2$ in place of $t$. As for Lemma \ref{min att}, we see that the sequence $\smash{(\boldsymbol{\eta}_{x,y}^{t,n})_{n\in\N}}$ is tight. Hence it converges weakly to some $\smash{\boldsymbol{\eta}_{x,y}^t} \in \Prob(\PS{t})$ along a subsequence which we do not relabel.

Let $p\in (1,\infty)$ arbitrary, and fix $\kkkkkk\in\Cont_\bdd(\mms\times\mms)$ with $\kkkkkk\leq\kkkkk$ on $\mms\times\mms$. Then by Hölder's inequality and the nonincreasingness of $\smash{t\mapsto\W_{p_n}^\kkkkk(\delta_x,\delta_y,t)}$ for large enough $n$, we obtain
\begin{align*}
&\Big(\int_{\PS{t}}\expo^{\int_0^{t} p\kkkkkk\left(\gamma_{r}^1,\gamma_{r}^2\right)/2\d r}\met^p\!\big(\gamma_{t}^1,\gamma_{t}^2\big)\d\boldsymbol{\eta}_{x,y}^t(\gamma)\Big)^{1/p}\\
&\qquad\qquad \leq \liminf_{n\to\infty}\Big(\int_{\PS{t}} \expo^{\int_0^tp\kkkkkk\left(\gamma_{r}^1,\gamma_{r}^2\right)/2\d r}\met^p\!\big(\gamma_{t}^1,\gamma_{t}^2\big)\d\boldsymbol{\eta}_{x,y}^{t,n}(\gamma)\Big)^{1/p}\\
&\qquad\qquad \leq \limsup_{n\to\infty}\Big(\int_{\PS{t}} \expo^{\int_0^t p_n\kkkkk\left(\gamma_{r}^1,\gamma_{r}^2\right)/2\d r}\met^{p_n}\!\big(\gamma_{t}^1,\gamma_{t}^2\big)\d\boldsymbol{\eta}_{x,y}^{t,n}(\gamma)\Big)^{1/p_n} \leq \met(x,y).
\end{align*}
Sending $p\to \infty$ and then approximating $\kkkkk$ from below by means of Lemma \ref{Le:kk approximation} gives 
\begin{equation*}
\met\!\big(\gamma_{t}^1,\gamma_{t}^2\big) \leq \expo^{-\int_0^t \kkkkk\left(\gamma_{r}^1\gamma_{r}^2\right)/2\d r}\,\met(x,y)\quad\text{for }\boldsymbol{\eta}_{x,y}^t\text{-a.e. }\gamma\in\PS{t}.
\end{equation*}
A measurable selection argument as in the proof of Lemma \ref{kernel} establishes the claim.
\end{proof}

The next goal is to obtain a measure which obeys such pathwise bound at \emph{every} initial and terminal time instance in, say, $[0,1]$. Indeed, this is the point where the main work has to be done.

\begin{thm}\label{Thm:Main construction} There exists a universally measurable map
\begin{equation*}
\boldsymbol{\mu}\colon\quad\mms\times\mms\quad\longrightarrow\quad \Prob(\PS{1})
\end{equation*}
such that for every $x,y \in \mms$, we have that the marginals of $\boldsymbol{\mu}_{x,y} := \boldsymbol{\mu}(x,y)$ are laws of Brownian motions, restricted to $[0,1]$, starting in $x$ and $y$, respectively, and that there exists a $\boldsymbol{\mu}_{x,y}$-negligible Borel set $E \subset \PS{1}$ such that
\begin{equation*}
\met\!\big(\gamma_t^1,\gamma_t^2\big) \leq \expo^{-\int_s^t \kkkkk\left(\gamma_r^1,\gamma_r^2\right)/2\d r}\met\!\big(\gamma_s^1,\gamma_s^2\big)\quad\text{for every }s,t\in [0,1]\text{ with }s\leq t
\end{equation*}
for all $\gamma \in \PS{1} \setminus E$.
\end{thm}

\begin{proof} The strategy relies on patching the laws obtained in the previous proposition together on small dyadic partitions of $[0,1]$. Denote by $\boldsymbol{\mu}^{2^{-n}}$ the map from Proposition \ref{Pr:Law at t} and define
\begin{equation*}
\boldsymbol{\mu}_{n,x,y} := \underbrace{\boldsymbol{\mu}^{2^{-n}}\circ\cdots\circ\boldsymbol{\mu}^{2^{-n}}}_{2^{n-1}\text{ kernels}} \circ\ \!\boldsymbol{\mu}^{2^{-n}}_{x,y}\in\Prob(\PS{1}),
\end{equation*}
that is, at every dyadic partition point of $[0,1]$ at scale $2^{-n}$, we attach a new random curve evolving according to the law obtained in Proposition \ref{Pr:Law at t} to the random endpoint of the previous curve. The marginals of $\boldsymbol{\mu}_{n,x,y}$ are the laws of Brownian motions on $\mms$, restricted to $[0,1]$, starting in $x$ and $y$, respectively. As in the proof of Lemma \ref{min att}, we may exhibit a subsequence, not relabeled in the sequel, weakly converging to some $\boldsymbol{\mu}_{x,y} \in \Prob(\PS{1})$.

The key point lies in proving that for every $s,t\in \Q\cap [0,1]$ with $s\leq t$, there exists a $\boldsymbol{\mu}_{x,y}$-negligible Borel set $E_{s,t} \subset \PS{1}$ such that, for every $\gamma\in\PS{1}\setminus E_{s,t}$,
\begin{equation}\label{Eq:s-t bound}
\met\!\big(\gamma_t^1,\gamma_t^2\big) \leq \expo^{-\int_s^t\kkkkk\left(\gamma_r^1,\gamma_r^2\right)/2\d r}\met\!\big(\gamma_s^1,\gamma_s^2\big).
\end{equation}
By continuity of curves, the desired requirements are then satisfied by the $\boldsymbol{\mu}_{x,y}$-null set
\begin{equation*}
E := \bigcup_{\substack{s,t\in \Q\cap[0,1],\\ s\leq t}} E_{s,t}.
\end{equation*}

Let $\kkkkkk\in\Cont_\bdd(\mms\times\mms)$ as above, i.e.~$\kkkkkk\leq \kkkkk$ on $\mms\times\mms$. Pick $s$ and $t$ as above and notice that the sequences $s_m := 2^{-m}\lfloor 2^ms\rfloor$ and $t_m := 2^{-m}\lfloor 2^mt\rfloor$ tend to $s$ and $t$, respectively, as $m \to \infty$. Fix $m\in\N$ and an arbitrary $n\geq m$. Given any $i \in\{1,\dots,2^n-1\}$, for every path $\widetilde{\gamma} \in \PS{2^{-n}}$ one gets 
\begin{align*}
\met\!\big(\gamma_{2^{-n}}^1,\gamma_{2^{-n}}^2\big) \leq \expo^{-\int_0^{2^{-n}} \kkkkkk\left(\gamma_r^1,\gamma_r^2\right)/2\d r}  \met\!\big(\widetilde{\gamma}_{2^{-n}}^1,\widetilde{\gamma}_{2^{-n}}^2\big)\quad\text{for }\boldsymbol{\mu}^{2^{-n}}_{\widetilde{\gamma}_{2^{-n}}^1,\widetilde{\gamma}_{2^{-n}}^2}\text{-a.e. }\gamma\in\PS{2^{-n}}.
\end{align*}
Observing that the dyadic partition of $[0,1]$ of step size $2^{-n}$ contains the one at scale $2^{-m}$ and then integrating the resulting $\boldsymbol{\mu}_{n,x,y}$-a.e.~valid estimate, truncated at large enough $C>0$, against an arbitrary nonnegative function $\phi\in\Cont_\bdd(\PS{1})$, we obtain
\begin{align*}
\int_{\PS{1}}\phi(\gamma) \met_C\!\big(\gamma_{t_m}^1,\gamma_{t_m}^2\big)\d\boldsymbol{\mu}_{n,x,y}(\gamma) \leq \int_{\PS{1}}\phi(\gamma)\,\expo^{-\int_{2^{-n}\lfloor 2^ns_m \rfloor}^{2^{-n}\lfloor 2^nt_m\rfloor} \kkkkkk\left(\gamma_r^1,\gamma_r^2\right)/2\d r}  \met_C\!\big(\gamma_{s_m}^1,\gamma_{s_m}^2\big)\d\boldsymbol{\mu}_{n,x,y}(\gamma),
\end{align*}
where $\met_C := \min\{\met,C\}$. Since $\kkkkkk$ is bounded, for all $m \in \N$ and every $\varepsilon > 0$, this yields
\begin{align*}
\int_{\PS{1}}\phi(\gamma) \met_C\!\big(\gamma_{t_m}^1,\gamma_{t_m}^2\big)\d\boldsymbol{\mu}_{n,x,y}(\gamma) &\leq \int_{\PS{1}} \phi(\gamma)\, \expo^{-\int_{s_m}^{t_m} \kkkkkk\left(\gamma_r^1,\gamma_r^2\right)/2\d r}  \met_C\!\big(\gamma_{s_m}^1,\gamma_{s_m}^2\big)\d\boldsymbol{\mu}_{n,x,y}(\gamma)\nonumber\\
&\qquad \qquad + \varepsilon \int_{\PS{1}} \phi(\gamma)  \met_C\!\big(\gamma_{s_m}^1,\gamma_{s_m}^2\big)\d\boldsymbol{\mu}_{n,x,y}(\gamma)
\end{align*}
for all large enough $n$. Letting $n\to\infty$, $\varepsilon \downarrow 0$ and then $C\to\infty$ in the previous estimate as well as extending the class of $\phi$ to nonnegative, bounded Borel functions by a routine approximation argument, we get
\begin{equation}\label{PE_m}
\met\!\big(\gamma_{t_m}^1,\gamma_{t_m}^2\big) \leq \expo^{-\int_{s_m}^{t_m}\kkkkkk\left(\gamma_r^1,\gamma_r^2\right)/2 \d r}\met\!\big(\gamma_{s_m}^1,\gamma_{s_m}^2\big)\quad\text{for }\boldsymbol{\mu}_{x,y}\text{-a.e. }\gamma\in\PS{1}.
\end{equation}
Let us now put
\begin{equation*}
\widetilde{E}_{s,t} := \bigcup_{m \in \N}\{\gamma \in \PS{1} : \gamma\text{ does not satisfy }\eqref{PE_m}\}
\end{equation*}
which clearly satisfies $\smash{\boldsymbol{\mu}_{x,y}\big[\widetilde{E}_{s,t}\big]} = 0$, and \eqref{Eq:s-t bound} holds on $\PS{1}\setminus \smash{\widetilde{E}_{s,t}}$ with $\kkkkkk$ in place of $\kkkkk$ by the convergences $s_m\to s$ and $t_m\to t$ as $m\to\infty$. Finally, denoting by $\kkkkk_n\in\Lip_\bdd(\mms)$ a sequence approximating $\smash{\kkkkk}$ from below as provided by Lemma \ref{Le:kk approximation}, the above reasoning gives Borel subsets $\smash{\widetilde{E}_{s,t}^n}$ of $\PS{1}$ such that $\smash{\boldsymbol{\mu}_{x,y}\big[\widetilde{E}_{s,t}^n\big] = 0}$ and 
\begin{equation*}
\met\!\big(\gamma_t^1,\gamma_t^2\big) \leq \expo^{-\int_s^t \kkkkk_n\left(\gamma_r^1,\gamma_r^2\right)/2\d r} \met\!\big(\gamma_s^1,\gamma_s^2\big)
\end{equation*}
for every $\gamma\in \PS{1}\setminus \smash{\widetilde{E}_{s,t}^n}$. Putting
\begin{equation*}
E_{s,t} := \bigcup_{n=1}^\infty \widetilde{E}_{s,t}^n,
\end{equation*}
we see that $\boldsymbol{\mu}_{x,y}\big[E_{s,t}\big] = 0$ and that \eqref{Eq:s-t bound} holds for all $\gamma \in\PS{1}\setminus E_{s,t}$ by monotone convergence. 

A similar argument and arguing as for Lemma \ref{kernel} shows that we can then select the obtained measures in a universally measurable way.
\end{proof}

\subsection[Extension to arbitrary initial distributions and time interval {$[0,\infty)$}]{Extension to arbitrary initial distributions and time interval {\boldmath{$[0,\infty)$}}}\label{Sec:Sec 6.2}

The cases of arbitrary initial distributions $\mu\in\Prob(\mms\times\mms)$ and an infinite time horizon are immediate given the construction in the proof of Theorem \ref{Thm:Main construction}. By iteratively composing copies of $\boldsymbol{\mu}$ with $\boldsymbol{\mu}\circ\mu$, we obtain a measure $\boldsymbol{\rho}_\mu\in\Prob(\Cont([0,\infty);\mms\times\mms))$ such that $(\eval_0)_\push\boldsymbol{\rho}_\mu = \mu$. The pathwise coupling properties on each interval $[n-1,n]$, $n\in\N$, which are inherited by $\boldsymbol{\mu}$ carry over to the entire space. As a result, we get the following.

\begin{thm} For all $\mu\in\Prob(\mms\times\mms)$ with marginals $\mu_1,\mu_2\in\Prob(\mms)$, the measure $\boldsymbol{\rho}_\mu$ constructed above satisfies the following properties: both its marginals coincide with the law of Brownian motions  on $\mms$ starting in $\mu_1$ and $\mu_2$, respectively, and for $\boldsymbol{\rho}_\mu$-a.e.~$\gamma\in\Cont([0,\infty);\mms\times\mms)$, we have
\begin{equation*}
\met\!\big(\gamma_t^1,\gamma_t^2\big) \leq \expo^{-\int_s^t \kkkkk\left(\gamma_r^1,\gamma_r^2\right)/2\d r}\met\!\big(\gamma_s^1,\gamma_s^2\big)\quad\text{for every }s,t\in [0,\infty)\text{ with } s \leq t.
\end{equation*}
\end{thm}

By considering the canonical process $\big(\B^1,\B^2\big)$ defined by $\B_t^1(\gamma) := \gamma_t^1$ and $\B_t^2(\gamma) := \gamma_t^2$ under the measure $\boldsymbol{\rho}_{\mu}$, we immediately obtain the assertion of Theorem \ref{pte to pcp}, which is just a stochastic rephrasing of the previous result.

\end{document}